\documentclass[abstract=on, 11pt, a4paper, pointlessnumbers]{scrartcl}

\usepackage[T1]{fontenc}

\usepackage{amsmath, amssymb, amsthm}
\usepackage{mathtools}
\usepackage{latexsym}

\usepackage{amsbsy}

\usepackage[automark, headsepline, footsepline, plainfootsepline]{scrlayer-scrpage}

\clearscrheadfoot

\clearscrplain
\ofoot[\pagemark]{\pagemark}
\ohead{\headmark}
\automark{section}

\usepackage[left=3cm, right=2.5cm, top=2.5cm, bottom=3.5cm]{geometry}

\usepackage{enumitem}

\usepackage{xcolor}

\usepackage[most]{tcolorbox}
\tcbuselibrary{listings,theorems}

\usepackage[singlespacing]{setspace}

\usepackage{tikz}
\usetikzlibrary{arrows,matrix,trees}
\usetikzlibrary{circuits}

\usepackage{graphicx}
\usepackage{pdfpages}

\usepackage[super]{nth}

\newcounter{countercheck}[subsection]

\theoremstyle{plain}
\swapnumbers
\newtheorem{theorem}[countercheck]{Theorem}
\newtheorem{proposition}[countercheck]{Proposition}
\newtheorem{lemma}[countercheck]{Lemma}
\newtheorem{corollary}[countercheck]{Corollary}

\theoremstyle{definition}

\newtheorem{definition}[countercheck]{Definition}
\newtheorem{convention}[countercheck]{Convention}

\theoremstyle{remark}

\usepackage{comment}

\renewcommand{\phi}{\varphi}
\renewcommand{\epsilon}{\varepsilon}
\newcommand{\NNN}{\mathbb{N}_0}
\newcommand{\NN}{\mathbb{N}}

\newcommand{\Cp}{\mathcal{C}_{\epsilon}}
\newcommand{\Cm}{\mathcal{C}_{-\epsilon}}
\newcommand{\C}{\mathcal{C}}
\renewcommand{\P}{\mathcal{P}}
\newcommand{\Cbf}{\textbf{\textup{C}}}

\DeclareMathOperator{\proj}{proj}

\DeclareMathOperator{\co}{\textup{(co)}}

\numberwithin{equation}{section} 
\title{Isometries of wall-connected twin buildings}
\author{Sebastian \textit{Bischof}\footnote{email: sebastian.bischof@math.uni-giessen.de} \and Bernhard \textit{M\"uhlherr}\footnote{email: bernhard.m.muehlherr@math.uni-giessen.de} \\ \\
	Mathematisches Institut, Arndtstra\ss e 2, 35392 Gie\ss en, Germany}

\date{\today}

\begin{document}

\maketitle

\begin{abstract}
	We introduce the notion of a wall-connected twin building and show that the local-to-global principle holds for these twin buildings.
	As each twin building satisfying Condition $\co$ (introduced in \cite{MR95}) is wall-connected, we obtain a strengthening of the main result of \cite{MR95} that covers also the thick irreducible affine twin buildings of rank at least $3$.
	
		
\end{abstract}

\section{Introduction}

In \cite{Ti74} Tits gave a complete classification of all thick irreducible spherical buildings of rank at least $3$.
The decisive tool in this classification is the extension theorem for local isometries
between two spherical buildings (Theorem $4.1.2$ in loc.\,cit.). Inspired by the paper
\cite{Ti87} on Kac-Moody groups over fields, Ronan and Tits introduced twin buildings (see \cite[$88/89$ and $89/90$]{Ti73-00} and \cite{Ti92}). Twin buildings
are natural generalizations of spherical buildings because there is an opposition relation on the set of its chambers. It was conjectured in \cite{Ti92}
that the extension theorem for local isometries can be generalized to
$2$-spherical twin buildings. This conjecture was confirmed in \cite{MR95} for
twin buildings that satisfy a technical condition (called Condition $\co$ in \cite{MR95})
and it was shown that Condition $\co$ is "almost always" satisfied.
More precisely, if the twin building in question has no rank two residues
isomorphic to $B_2(2),G_2(2),G_2(3)$ or $^2F_4(2)$, then the extension theorem
holds (see Corollary $1.7$ in \cite{MR95}).

It seemed first that Condition $\co$ was just a convenient
assumption for making the ideas in \cite{MR95} work, but that
the extension theorem should hold without any additional hypothesis. After a while, however,
there were serious doubts about this (see Paragraph $2.3$ in \cite[$97/98$]{Ti73-00})
and the question about the validity of the local-to-global principle for all $2$-spherical
buildings remained an open problem. It is particularly unsatisfying that it is even not known whether the extension theorem holds for twin buildings of affine type. It was observed by Abramenko and Mühlherr that the arguments in \cite{MR95} can be modified to treat some cases in which Condition $\co$ does not hold. But these modifications were not good enough to prove the extension theorem for all affine twin buildings.

In this paper we introduce a condition for twin buildings that we call {\sl wall-connectedness}. It is inspired by the content of \cite{DMVM11} and its definition (given in Definition \ref{Definition: wall-connected}) is somewhat technical. It turns out that each twin building satisfying Condition $\co$ is wall-connected but that the converse is not true. The main goal of this paper is to provide the following improvement of the main result in \cite{MR95}.

\medskip
\noindent \textbf{Main result:} The extension theorem holds for wall-connected twin buildings.

\medskip For a precise statement of our main result we refer to Corollary \ref{Cor: Main result}. It turns out that all $3$-spherical twin buildings and all irreducible affine twin buildings of rank at least $3$ are wall-connected (see Section \ref{Section: Wall-connected twin buildings}). Thus, our main result yields the following:

\medskip
\noindent \textbf{Corollary:} The extension theorem holds for all $3$-spherical twin buildings and all irreducible affine twin buildings of rank at least $3$.

\subsection*{Content of the paper}

In Section $2$ we fix notation and state some results about parallel residues in a building. In Section $3$ we give the definition of a twin building and prove some elementary facts which we will need later. This section is also to be understood as a preliminary section. The first part of the next section is concerned with compatible path's as defined in \cite{DMVM11}. In the second part of this section we define $P$-compatible path's which generalizes compatible path's to the situation of twin buildings. Later on we prove some result about them. In particular, our proof of the extension theorem relies heavily on Proposition \ref{projectiontransitive}. At the end of this section we give the definition of a wall-connected twin building. Section $5$ is divided in three parts. In the first part we state the definition of an isometry and some basic results about them. A crucial lemma is Lemma \ref{Mu97Lemma4.3}. We will use this lemma in combination with Proposition \ref{projectiontransitive} to prove the extension theorem for wall-connected twin buildings. The main step is Proposition \ref{Propphiswelldefined}.

The rest of the paper is concerned essentially with the fact that affine twin buildings of rank at least $3$ are wall-connected.

\renewcommand{\abstractname}{Acknowledgement}
\begin{abstract}
	We thank Richard Weiss for communicating us the proof of Proposition \ref{Wurzelgruppen}.
\end{abstract}

\section{Preliminaries}

\subsection*{Coxeter system}

A \textit{Coxeter system} is a pair $(W, S)$ consisting of a group $W$ and a set $S \subseteq W$ of generators of $W$ of order $2$ such that the set $S$ and the relations $\left(st \right)^{m_{st}}$ for all $s,t\in S$ constitute a presentation of $W$, where $m_{st}$ denotes the order of $st$ in $W$ for all $s, t \in S$.

Let $(W, S)$ be a Coxeter system and let $\ell: W \to \NN, w \mapsto \min\{ k\in \NNN \mid \exists s_1, \ldots, s_k \in S: w = s_1 \cdots s_k \}$ denote the corresponding length function. The \textit{Coxeter diagram} corresponding to $(W, S)$ is the labeled graph $(S, E(S))$, where $E(S) = \{ \{s, t \} \mid m_{st}>2 \}$ and where each edge $\{s,t\}$ is labeled by $m_{st}$ for all $s, t \in S$. We call the Coxeter diagram \textit{irreducible}, if the underlying graph is connected, and we call it \textit{$2$-spherical}, if $m_{st} <\infty$ for all $s,t \in S$. The \textit{rank} of a Coxeter diagram is the cardinality of the set of its vertices. It is well-known that the pair $(\langle J \rangle, J)$ is a Coxeter system (cf. \cite[Ch. IV, §$1$ Th\'eor\`eme $2$]{Bo68}). We call $J \subseteq S$ \textit{spherical} if $\langle J \rangle$ is finite. Given a spherical subset $J$ of $S$, there exists a unique element of maximal length in $\langle J \rangle$, which we denote by $r_J$ (cf. \cite[Corollary $2.19$]{AB08}); moreover, $r_J$ is an involution.

\begin{convention}
	From now on we let $(W, S)$ be a Coxeter system of finite rank.
\end{convention}

\begin{lemma}\label{NearlyConditionF}
	Let $(W, S)$ be a Coxeter system and let $w\in W, s, t\in S$ be such that $\ell(sw) = \ell(w) -1 = \ell(wt)$. Then either $\ell(swt) = \ell(w) -2$ or $swt = w$.
\end{lemma}
\begin{proof}
	We put $w' := sw$. Then $\ell(sw') = \ell(w) = \ell(w') +1$. We assume that $\ell(swt) \neq \ell(w) -2$. Then $\ell(swt) = \ell(w)$ and hence $\ell(w't) = \ell(swt) = \ell(w) = \ell(w') +1$. Using Condition $(\mathbf{F})$ of \cite{AB08} on page $79$ we obtain either $\ell(sw't) = \ell(w') +2$ or $sw't = w'$. Since $\ell(sw't) = \ell(wt) = \ell(w) -1 = \ell(w')$ we have $wt = sw't = w' =  sw$ and the claim follows.
\end{proof}

\subsection*{Buildings}

A \textit{building of type $(W, S)$} is a pair $\Delta = (\C, \delta)$ where $\C$ is a non-empty set and where $\delta: \C \times \C \to W$ is a \textit{distance function} satisfying the following axioms, where $x, y\in \C$ and $w = \delta(x, y)$:
\begin{enumerate}[label=(Bu\arabic*), leftmargin=*]
	\item $w = 1_W$ if and only if $x=y$;
	
	\item if $z\in \C$ satisfies $s := \delta(y, z) \in S$, then $\delta(x, z) \in \{w, ws\}$, and if, furthermore, $\ell(ws) = \ell(w) +1$, then $\delta(x, z) = ws$;
	
	\item if $s\in S$, there exists $z\in \C$ such that $\delta(y, z) = s$ and $\delta(x, z) = ws$.
\end{enumerate}
The \textit{rank} of $\Delta$ is the rank of the underlying Coxeter system. The elements of $\C$ are called \textit{chambers}. Given $s\in S$ and $x, y\in \C$, then $x$ is called \textit{$s$-adjacent} to $y$, if $\delta(x, y) = \langle s \rangle$. The chambers $x, y$ are called \textit{adjacent}, if they are $s$-adjacent for some $s\in S$. A \textit{gallery} joining $x$ and $y$ is a sequence $(x = x_0, \ldots, x_k = y)$ such that $x_{l-1}$ and $x_l$ are adjacent for any $1 \leq l \leq k$; the number $k$ is called the \textit{length} of the gallery. For any two chambers $x$ and $y$ we put $\ell_{\Delta}(x, y) := \ell(\delta(x, y))$.

Given a subset $J \subseteq S$ and $x\in \C$, the \textit{$J$-residue} of $x$ is the set $R_J(x) := \{y \in \C \mid \delta(x, y) \in \langle J \rangle \}$. Each $J$-residue is a building of type $(\langle J \rangle, J)$ with the distance function induced by $\delta$ (cf. \cite[Corollary $5.30$]{AB08}). A \textit{residue} is a subset $R$ of $\C$ such that there exist $J \subseteq S$ and $x\in \C$ with $R = R_J(x)$. It is a basic fact that the subset $J$ is uniquely determined by the set $R$; it is called the \textit{type} of $R$ and the \textit{rank} of $R$ is defined to be the cardinality of $J$. A residue is called \textit{spherical} if its type is a spherical subset of $S$. Let $R$ be a spherical residue of type $J$ and let $x, y \in R$. Then $x, y$ are called \textit{opposite in $R$} if $\delta(x, y) = r_J$. If $R=\C$, we say for short that $x, y$ are \textit{opposite}. If $(W, S)$ is spherical we call two residues $R_1$ of type $J_1$ and $R_2$ of type $J_2$ \textit{opposite} if $R_1$ contains a chamber opposite to a chamber of $R_2$ and if $J_1 = r_S J_2 r_S$.

A \textit{panel} is a residue of rank $1$. An \textit{$s$-panel} is a panel of type $\{s\}$ for some $s\in S$. The building $\Delta$ is called \textit{thick}, if each panel of $\Delta$ contains at least three chambers.

Given $x\in \C$ and $k \in \NNN$ then $E_k(x)$ denotes the union of all residues of rank at most $k$ containing $x$. It is a fact, that the set $E_k(x)$ determines the chamber $x$ uniquely, if $k < \vert S \vert$.

For $x\in \C$ and any $J$-residue $R \subseteq \C$ there exists a unique chamber $z\in R$ such that $\ell_{\Delta}(x, y) = \ell_{\Delta}(x, z) + \ell_{\Delta}(z, y)$ holds for any $y\in R$ (cf. \cite[Proposition $5.34$]{AB08}). The chamber $z$ is called the \textit{projection of $x$ onto $R$} and is denoted by $\proj_R x$. Moreover, if $z = \proj_R x$ we have $\delta(x, y) = \delta(x, z) \delta(z, y)$ for each $y\in R$.

Let $R, Q$ be two residues. Then we define the mapping $\proj^R_Q: R \to Q, x \mapsto \proj_Q x$ and we put $\proj_Q R := \{ \proj_Q r \mid r\in R \}$. The residues $R, Q$ are called \textit{parallel} if $\proj_R Q = R$ and $\proj_Q R = Q$.

\begin{lemma}\label{parallelresidues}
	Two residues $R, Q$ are parallel if and only if $\proj^R_Q$ and $\proj^Q_R$ are bijections inverse to each other.
\end{lemma}
\begin{proof}
	One implication is obvious; the other is \cite[Proposition $21.10$]{MPW15}.
\end{proof}

\begin{lemma}\label{DMVMLemma14}
	Let $P_1$ and $P_2$ be two parallel panels of type $s_1$ and $s_2$, respectively. Then $s_2 = w^{-1}s_1 w$, where $w := \delta(x, \proj_{P_2} x)$ does not depend on the choice of $x$ in $P_1$.
	
	Conversely, if $x$ and $y$ are chambers with $\delta(x, y) = w$, where $w$ satisfies $s_2 = w^{-1} s_1 w$ and $\ell(s_1 w) = \ell(w) +1$, then the $s_1$-panel on $x$ is parallel to the $s_2$-panel on $y$.
\end{lemma}
\begin{proof}
	This is \cite[Lemma $14$]{DMVM11}.
\end{proof}

Let $P_1$ and $P_2$ be two parallel panels. Then, by the previous lemma, $\delta(x, \proj_{P_2} x)$ does not depend on the choice of $x\in P_1$. Thus we define $\delta(P_1, P_2) := \delta(x, \proj_{P_2} x)$, where $x$ is a chamber in $P_1$.

\begin{lemma}\label{DMVMProposition4}
	Let $R$ be a spherical $J$-residue and let $R_1, R_2$ be two residues in $R$, which are opposite in $R$. Then $R_1$ and $R_2$ are parallel.
\end{lemma}
\begin{proof}
	This is a consequence of \cite[Proposition $21.24$]{MPW15}.
\end{proof}

\begin{lemma}\label{DMVMLemma18}
	Let $R$ be a rank $2$ residue and let $P, Q$ be two parallel panels contained in $R$. Then either $P = Q$ or $R$ is spherical and $P$ and $Q$ are opposite in $R$. In particular, if $P \neq Q$ and $J$ is the type of $R$, we have $\ell(\delta(P, Q)) = \ell(r_J) -1$.
\end{lemma}
\begin{proof}
	This is \cite[Lemma $18$]{DMVM11}.
\end{proof}

A subset $\Sigma \subseteq \C$ is called \textit{convex} if $\proj_P c \in \Sigma$ for every $c\in \Sigma$ and every panel $P \subseteq \C$ which meets $\Sigma$. A subset $\Sigma \subseteq \C$ is called \textit{thin} if $P \cap \Sigma$ contains exactly two chambers for every panel $P \subseteq \C$ which meets $\Sigma$. An \textit{apartment} is a non-empty subset $\Sigma \subseteq \C$, which is convex and thin.

\section{Twin buildings}

Let $\Delta_+ = (\C_+, \delta_+), \Delta_- = (\C_-, \delta_-)$ be two buildings of the same type $(W, S)$. A \textit{codistance} (or a \textit{twinning}) between $\Delta_+$ and $\Delta_-$ is a mapping $\delta_* : \left( \C_+ \times \C_- \right) \cup \left( \C_- \times \C_+ \right) \to W$ satisfying the following axioms, where $\epsilon \in \{+,-\}, x\in \Cp, y\in \Cm$ and $w=\delta_*(x, y)$:
\begin{enumerate}[label=(Tw\arabic*), leftmargin=*]
	\item $\delta_*(y, x) = w^{-1}$;
	
	\item if $z\in \Cm$ is such that $s := \delta_{-\epsilon}(y, z) \in S$ and $\ell(ws) = \ell(w) -1$, then $\delta_*(x, z) = ws$;
	
	\item if $s\in S$, there exists $z\in \Cm$ such that $\delta_{-\epsilon}(y, z) = s$ and $\delta_*(x, z) = ws$.
\end{enumerate}
A \textit{twin building of type $(W, S)$} is a triple $\Delta = (\Delta_+, \Delta_-, \delta_*)$ where $\Delta_+ = (\C_+, \delta_+), \Delta_- = (\C_-, \delta_-)$ are buildings of type $(W, S)$ and where $\delta_*$ is a twinning between $\Delta_+$ and $\Delta_-$.

\begin{lemma}\label{AB08Lemma5.139}
	Given $\epsilon \in \{+,-\}, x \in \Cp, y \in \Cm$ and let $w = \delta_*(x, y)$. Then for any $y' \in \Cm$ with $\delta_{-\epsilon}(y, y') =s \in S$ we have $\delta_*(x, y') \in \{w, ws\}$.
\end{lemma}
\begin{proof}
	This follows similar to \cite[Lemma $5.139$]{AB08}.
\end{proof}

We put $\C := \C_+ \cup \C_-$ and define the distance function $\delta: \C \times \C \to W$ by setting $\delta(x, y) := \delta_+(x, y)$ (resp. $\delta_-(x, y), \delta_*(x, y)$) if $x,y\in \C_+$ (resp. $x, y \in \C_-, (x, y) \in \Cp \times \Cm$ for some $\epsilon \in \{+,-\}$).

Given $x, y \in \C$ then we put $\ell(x, y) := \ell(\delta(x, y))$. If $\epsilon \in \{+,-\}$ and $x, y \in \Cp$, then we put $\ell_{\epsilon}(x, y) := \ell(\delta_{\epsilon}(x, y))$ and for $(x, y) \in \Cp \times \Cm$ we put $\ell_*(x, y) := \ell(\delta_*(x, y))$.

Let $\epsilon \in \{+,-\}$. For $x\in \Cp$ we put $x^{\mathrm{op}} := \{ y\in \Cm \mid \delta_*(x, y) = 1_W \}$. It is a direct consequence of (Tw$1$) that $y\in x^{\mathrm{op}}$ if and only if $x\in y^{\mathrm{op}}$ for any pair $(x, y) \in \Cp \times \Cm$. If $y\in x^{\mathrm{op}}$ then we say that $y$ is \textit{opposite} to $x$ or that \textit{$(x, y)$ is a pair of opposite chambers}.

A \textit{residue} (resp. \textit{panel}) of $\Delta$ is a residue (resp. panel) of $\Delta_+$ or $\Delta_-$; given a residue $R\subseteq \C$ then we define its type and rank as before. Two residues $R,T \subseteq \C$ in different halves are called \textit{opposite} if they have the same type and if there exists a pair of opposite chambers $(x, y)$ such that $x\in R, y\in T$. The twin building $\Delta$ is called \textit{thick} if $\Delta_+$ and $\Delta_-$ are thick.

Let $\epsilon \in \{+,-\}$, let $J$ be a spherical subset of $S$ and let $R$ be a $J$-residue of $\Delta_{\epsilon}$. Given a chamber $x\in \Cm$ then there exists a unique chamber $z\in R$ such that $\ell_*(x, y) = \ell_*(x, z) - \ell_{\epsilon}(z, y)$ holds for any chamber $y\in R$ (cf. \cite[Lemma $5.149$]{AB08}). The chamber $z$ is called the \textit{projection of $x$ onto $R$} and is denoted by $\proj_R x$. Moreover, if $z = \proj_R x$ we have $\delta_*(x, y) = \delta_*(x, z)\delta_{\epsilon}(z, y)$ for each $y\in R$.

\begin{lemma}\label{residueprojectionabsorbtion}
	Let $R_1 \subseteq R_2$ be two spherical residues of $\Delta$ and let $x\in \C$. Then $\proj_{R_1} x = \proj_{R_1} \proj_{R_2} x$.
\end{lemma}
\begin{proof}
	Let $r\in R_1$. Following \cite[Proposition $2$]{DS87} we compute the following, where we have $'+'$ if $x, R_1, R_2$ are in the same half, and $'-'$ if $x$ and $R_1, R_2$ are in different halves: 
	\begin{align*}
	\ell(x, r) &= \ell(x, \proj_{R_2} x) \pm \ell(\proj_{R_2} x, r) \\
	&= \ell(x, \proj_{R_2} x) \pm \left( \ell(\proj_{R_2} x, \proj_{R_1} \proj_{R_2} x) + \ell(\proj_{R_1} \proj_{R_2} x, r) \right) \\
	&= \ell(x, \proj_{R_1} \proj_{R_2} x) \pm \ell(\proj_{R_1} \proj_{R_2} x, r)
	\end{align*}
	Since this holds for any $r\in R_1$, the uniqueness of $\proj_{R_1} x$ yields the claim.
\end{proof}

\begin{lemma}\label{Mu97Lemma3.4}
	Let $\epsilon \in \{+, -\}$ and let $R \subseteq \C_{\epsilon}$ and $T \subseteq \C_{-\epsilon}$ be two opposite residues of spherical type $J \subseteq S$. Then for any pair $(x, y) \in R \times T$ the following are equivalent:
	\begin{enumerate}[label=(\roman*), leftmargin=*]
		\item $\proj_T x = y$;
		\item $\delta_*(x, y) = r_J$;
		\item $\proj_R y = x$.
	\end{enumerate}
\end{lemma}
\begin{proof}
	This is \cite[Lemma $3.4$]{BCM21}.
\end{proof}

Let $\epsilon \in \{+,-\}$ and let $R \subseteq \Cp, T \subseteq \Cm$ be spherical residues. Then we define the mapping $\proj^R_T: R \to T, x \mapsto \proj_T x$ and we put $\proj_T R := \{ \proj_T r \mid r\in R \}$. The residues $R$ and $T$ are called \textit{parallel} if $\proj_R T = R$ and $\proj_T R = T$.

\begin{lemma}\label{AB08Exercise5.168+projbijection+Mu97Lemma3.5}
	Let $\epsilon \in \{+,-\}, R \subseteq \Cp, T \subseteq \Cm$ be two spherical residues.
	\begin{enumerate}[label=(\alph*), leftmargin=*]
		\item $\proj_T R$ is a spherical residue in $T$.
		
		\item $R$ and $T$ are parallel if and only if $\proj^R_T$ and $\proj^T_R$ are bijections inverse to each other.
		
		\item  $\proj_R T$ and $\proj_T R$ are parallel.
		
		\item If $R$ and $T$ are opposite then they are parallel.
	\end{enumerate} 
\end{lemma}
\begin{proof}
	Assertion $(a)$ is \cite[Exercise $5.168$]{AB08}. Let $x\in \proj_R T$. Then there exists $y\in T$ such that $x = \proj_R y$. Note that $\ell_*(y, x) = \ell_*(\proj_T x, x) - \ell_{-\epsilon}(y, \proj_T x)$. Since $\ell_*(c, d) \geq \ell_*(c, e) - \ell_{-\epsilon}(e, d)$ for any $c\in \C_{\epsilon}$ and $d, e\in \C_{-\epsilon}$ the following hold:
	\begin{align*}
		\ell_*(y, x) - \ell_{\epsilon}(x, \proj_R \proj_T x) &= \ell_*(y, \proj_R \proj_T x) \\
		&\geq \ell_*(\proj_T x, \proj_R \proj_T x) - \ell_{-\epsilon}(y, \proj_T x) \\
		&\geq \ell_*(\proj_T x, x) - \ell_{-\epsilon}(y, \proj_T x) \\
		&= \ell_*(y, x).
	\end{align*}
	This implies $\proj_R \proj_T x = x$ and the restriction of the projection mappings are bijections inverse to each other. 
	
	One implication in Assertion $(b)$ is obvious. For the other we note that $\proj_R T = R$ and $\proj_T R = T$ and we proved that the restriction of the projection mappings are bijections inverse to each other. Assertion $(c)$ is now a consequence of Assertion $(b)$ and Assertion $(d)$ is \cite[Proposition $(4.3)$]{Ro00}.
\end{proof}

\begin{lemma}\label{GeneralizationDMVMLemma13}
	Let $\epsilon \in \{+,-\}, P \subseteq \Cp, Q \subseteq \Cm$ be two panels. Then $P, Q$ are parallel if and only if $\vert \proj_P Q \vert \geq 2$.
\end{lemma}
\begin{proof}
	We follow the ideas of \cite[Lemma $13$]{DMVM11}. If $P, Q$ are parallel, the claim follows. Therefore let $\vert \proj_P Q \vert \geq 2$. Since $\proj_P Q$ is a residue contained in $P$ by Lemma \ref{AB08Exercise5.168+projbijection+Mu97Lemma3.5}$(a)$, we have $\proj_P Q = P$. By Lemma \ref{AB08Exercise5.168+projbijection+Mu97Lemma3.5}$(c)$ the residues $\proj_Q P$ and $\proj_P Q = P$ are parallel. Thus we have $\vert \proj_Q P \vert = \vert P \vert \geq 2$. Using the same argument we obtain $\proj_Q P = Q$ and the panels $P$ and $Q$ are parallel.
\end{proof}

\begin{lemma}\label{GeneralizationDMVM11Lemma17}
	Let $\epsilon \in \{+,-\}$, let $P \subseteq \Cp, Q \subseteq \Cm$ be two parallel panels and let $R$ be a spherical residue containing $Q$. Then $\proj_R P$ is a panel parallel to both $P$ and $Q$.
\end{lemma}
\begin{proof}
	For a proof see \cite[Lemma $17$]{DMVM11}. We note that the facts which are used in \cite{DMVM11} for buildings are proved in this paper for twin buildings.
\end{proof}

Let $\Sigma_+ \subseteq \C_+$ and $\Sigma_- \subseteq \C_-$ be apartments of $\Delta_+$ and $\Delta_-$, respectively. Then the set $\Sigma := \Sigma_+ \cup \Sigma_-$ is called \textit{twin apartment} if $\vert x^{\mathrm{op}} \cap \Sigma \vert = 1$ for each $x\in \Sigma$. If $(x, y)$ is a pair of opposite chambers, then there exists a unique twin apartment $A(x, y) = \{ z\in \C \mid \delta(x, z) = \delta(y, z) \} = \{ z\in \C \mid \ell(z, x) = \ell(z, y) \}$ containing $x$ and $y$ (cf. \cite[Exercise $5.187$, Proposition $5.179(1)$]{AB08}). For $\epsilon \in \{+,-\}$ we put $A_{\epsilon}(x, y) := A(x, y) \cap \Cp$. Furthermore, for any two chambers there exists a twin apartment containing them (cf. \cite[Proposition $5.179(3)$]{AB08}).

\begin{lemma}\label{Twinapartment}
	Let $(x, y)$ be a pair of opposite chambers and let $P \subseteq \C$ be a panel which meets $A(x, y)$. Then $A(x, y) \cap P = \{ \proj_P x \neq \proj_P y \}$.
\end{lemma}
\begin{proof}
	We have $\proj_P x, \proj_P y \in A(x, y) \cap P$ (cf. \cite[Lemma $5.173$ ($6$)]{AB08}). Since $A_{\epsilon}(x, y)$ is thin for each $\epsilon \in \{+,-\}$, we obtain $\vert A(x, y) \cap P \vert = 2$. We assume that $\proj_P x = \proj_P y$. Then there exists a chamber $\proj_P x \neq z \in A(x, y) \cap P$. We can assume that $y$ is in the same half of the twin building as the panel $P$. This implies
	\begin{align*}
	\ell(y, z) = \ell(x, z) = \ell(x, \proj_P x) - \ell(\proj_P x, z) = \ell(y, \proj_P y) - 1 = \ell(y, z) -2.
	\end{align*}
	This yields a contradiction and the claim follows.
\end{proof}

\begin{lemma}\label{Lemadjacent}
	Let $\epsilon \in \{+,-\}, s\in S$ and let $\Delta$ be a thick twin building. Then for all $x, y \in \Cm$ the following properties are equivalent:
	\begin{enumerate}[label=(\roman*), leftmargin=*]
		\item $\delta_{-\epsilon}(x, y) \in \langle s \rangle$;
		\item $\forall z\in x^{\mathrm{op}} \exists z' \in y^{\mathrm{op}}: \delta_{\epsilon}(z, z') = s$.
	\end{enumerate}
\end{lemma}
\begin{proof}
	If $x=y$ the claim follows because of the thickness of the twin building. Let $\delta_{-\epsilon}(x, y) = s$ and $z\in x^{\mathrm{op}}$. Then $\delta_*(z, y) \in \{1_W, s\}$ by Lemma \ref{AB08Lemma5.139}. If $\delta_*(z, y) = 1_W$, the claim follows because of the thickness again. Let $\delta_*(z, y) = s$. Then $\delta_*(y, z) = s$ and we obtain $z' \in \C_{-\epsilon}$ such that $\delta_{-\epsilon}(z, z') = s$ and $\delta_*(y, z') = 1_W$ by (Tw$3$). Therefore we have $(ii)$.
	
	Now we assume $(ii)$. There exists a twin apartment $\Sigma$ containing $x$ and $y$. Let $w = \delta_{-\epsilon}(y, x)$ and let $z \in \Sigma$ be the unique chamber which is opposite to $x$. This implies $\Sigma = A(x, z)$ and we obtain $\delta_*(y, z) = \delta_{-\epsilon}(y, x) = w$. By condition there exists a chamber $z' \in y^{\mathrm{op}}$ such that $\delta_{-\epsilon}(z, z') = s$. Then $1_W = \delta_*(y, z') \in \{ w, ws \}$ by Lemma \ref{AB08Lemma5.139}. The claim follows.
\end{proof}

\section{Wall-connected twin buildings}

\subsection*{Compatible path's}

Let $\Delta = (\C, \delta)$ be a thick building of type $(W, S)$ and let $\Gamma$ be the graph whose vertices are the panels of $\Delta$ and in which two panels form an edge if and only if there exists a rank $2$ residue in which the two panels are opposite. For two adjacent panels $P, Q$, there exists a unique rank $2$ residue containing $P$ and $Q$, that will be denoted by $R(P, Q)$. A path $\gamma = (P_0, \ldots, P_k)$ in $\Gamma$ is called \textit{compatible} if $\proj_{R(P_{i-1}, P_i)} P_0 = P_{i-1}$ for all $1 \leq i \leq k$. The number $k$ is the \textit{length} of that path $\gamma$. The sequence $(J_1, \ldots, J_k)$ where $J_i$ is the type of $R(P_{i-1}, P_i)$ will be called the \textit{type} of $\gamma$.

We obtain $\proj_{R(P_{i-1}, P_i)} x = \proj_{P_{i-1}} x$ for all $x\in P_0$ and $1 \leq i \leq k$, since $\proj_{R(P_{i-1}, P_i)} x \in P_{i-1}$. Furthermore, we obtain $\proj_{P_i} x = \proj_{P_i} \proj_{P_{i-1}} x$ for any $1 \leq i \leq k$ by Lemma \ref{residueprojectionabsorbtion}.

\begin{lemma}\label{DMVMLemma19}
	Two panels are parallel if and only if there exists a compatible path from one to the other.
\end{lemma}
\begin{proof}
	This is \cite[Lemma $19$]{DMVM11}.
\end{proof}

\begin{proposition}\label{DMVMProposition9}
	Let $(P_0, \ldots, P_k)$ be a compatible path. Then the following hold:
	\begin{enumerate}[label=(\alph*)]
		\item $\proj_{P_k}^{P_0} = \proj_{P_k}^{P_i} \circ \proj_{P_i}^{P_0}$ for any $0 \leq i \leq k$.
		
		\item $\delta(P_0, P_k) = \delta(P_0, P_i)\delta(P_i, P_k)$ for any $0 \leq i \leq k$;
		
		\item $\ell(\delta(P_0, P_k)) = \ell(\delta(P_0, P_i)) + \ell(\delta(P_i, P_k))$ for any $0 \leq i \leq k$.
		
		\item $(P_k, \ldots, P_0)$ is a compatible path.
	\end{enumerate}
\end{proposition}
\begin{proof}
	At first we prove Assertion $(c)$ by induction on $k$. For $k=0$ there is nothing to show. Thus we let $k>0$ and $x\in P_0$. Then $\ell_{\Delta}(x, \proj_{P_k} x) = \ell_{\Delta}(x, \proj_{R(P_{k-1}, P_k)} x) + \ell_{\Delta}(\proj_{R(P_{k-1}, P_k)} x, \proj_{P_k} x)$. Moreover, we have $\proj_{R(P_{k-1}, P_k)} x = \proj_{P_{k-1}} x$ and $\proj_{P_k} x = \proj_{P_k} \proj_{R(P_{k-1}, P_k)} x = \proj_{P_k} \proj_{P_{k-1}} x$. This implies 
	\[ \ell_{\Delta}(x, \proj_{P_k} x) = \ell_{\Delta}(x, \proj_{P_{k-1}} x) + \ell(\delta(P_{k-1}, P_k)) = \ell(\delta(P_0, P_{k-1})) + \ell(\delta(P_{k-1}, P_k)) \]
	Using induction, we have $\ell(\delta(P_0, P_{k-1})) = \ell(\delta(P_0, P_i)) + \ell(\delta(P_i, P_{k-1}))$ for any $1 \leq i \leq k-1$. We deduce that $\ell(\delta(P_0, P_k)) \geq \ell(\delta(P_0, P_i)) + \ell(\delta(P_i, P_k))$. In particular, we obtain 
	\begin{align*}
		\ell_{\Delta}(x, \proj_{P_k}x) &\geq \ell_{\Delta}(x, \proj_{P_i} x) + \ell_{\Delta}(\proj_{P_i}x, \proj_{P_k} \proj_{P_i} x) \geq \ell_{\Delta}(x, \proj_{P_k} x)
	\end{align*}
	This finishes the proof of Assertion $(c)$. Now Assertion $(b)$ is a direct consequence of Assertion $(c)$ and Assertion $(a)$ follows from Assertion $(c)$ and the uniqueness of the projection chamber.
	Note that $(a)$ also implies that $P_i$ and $P_k$ are parallel. For Assertion $(d)$ we use Lemma \ref{parallelresidues} and the equation of the projection mappings of Assertion $(a)$ and compute the following for each $0 \leq i \leq j \leq k$:
	\begin{align*}
	\proj_{P_i}^{P_k} = \proj_{P_i}^{P_0} \circ \proj_{P_0}^{P_k} = \left( \proj_{P_i}^{P_j} \circ \proj_{P_j}^{P_0} \right) \circ \proj_{P_0}^{P_k} = \proj_{P_i}^{P_j} \circ \proj_{P_j}^{P_k}.
	\end{align*}
	We have to show that $\proj_{R(P_{i-1}, P_i)} P_k = P_i$. For that we show $\proj_{R(P_{i-1}, P_i)} x = \proj_{P_i} x$ for each $x\in P_k$. Let $x\in P_k$ and let $r_i := \proj_{R(P_{i-1}, P_i)} x, p_i := \proj_{P_i} x$ and $p_{i-1} := \proj_{P_{i-1}} x$. Using Assertion $(c)$ and the fact that $\proj_{P_{i-1}} p_i = p_{i-1}$ we have 
	\begin{align*}
		\ell_{\Delta}(x, p_{i-1}) &= \ell(\delta(P_{i-1}, P_k)) \\
		&= \ell(\delta(P_0, P_k)) - \ell(\delta(P_0, P_{i-1})) \\
		&= \ell(\delta(P_0, P_i)) + \ell(\delta(P_i, P_k)) - \ell(\delta(P_0, P_{i-1})) \\
		&= \ell(\delta(P_k, P_i)) + \ell(\delta(P_i, P_{i-1})) \\
		&= \ell_{\Delta}(x, p_i) + \ell_{\Delta}(p_i, \proj_{P_{i-1}} p_i) \\
		&= \ell_{\Delta}(x, r_i) + \ell_{\Delta}(r_i, p_i) + \ell_{\Delta}(p_i, p_{i-1}).
	\end{align*}
	Since $\ell_{\Delta}(r_i, p_{i-1}) \leq \ell(r_{J_i}) -1 = \ell(\delta(P_{i-1}, P_i)) = \ell(\delta(P_i, P_{i-1})) = \ell_{\Delta}(p_i, p_{i-1})$, where $J_i$ is the type of the residue $R(P_{i-1}, P_i)$, we obtain
	\begin{align*}
		\ell_{\Delta}(x, r_i) + \ell_{\Delta}(r_i, p_i) + \ell_{\Delta}(p_i, p_{i-1}) &= \ell_{\Delta}(x, p_{i-1}) \\
		&= \ell_{\Delta}(x, r_i) + \ell_{\Delta}(r_i, p_{i-1}) \\
		&\leq \ell_{\Delta}(x, r_i) + \ell_{\Delta}(p_i, p_{i-1}).
	\end{align*}
	This yields $r_i = p_i$. Since $P_i, P_k$ are parallel, we obtain $\proj_{R(P_{i-1}, P_i)} P_k = \{ \proj_{R(P_{i-1}, P_i)} x \mid x \in P_k \} = \{ \proj_{P_i} x \mid x \in P_k \} = P_i$ and the claim follows.
\end{proof}

\begin{lemma}\label{DMVM11Lemma26}
	Let $s\in S$ and let $w\in W$ be such that $w^{-1}sw \in S$ and $\ell(sw) = \ell(w) +1$. Let $P, P'$ be $s$-panels and $Q, Q'$ be $w^{-1}sw$-panels such that $\delta(P, Q) = w = \delta(P', Q')$. If $(J_1, \ldots, J_k)$ is the type of a compatible path from $P$ to $Q$, then there exists a compatible path from $P'$ to $Q'$ of the same length and type.
\end{lemma}
\begin{proof}
	This is \cite[Lemma $26$]{DMVM11}.
\end{proof}

\subsection*{$P$-compatible path's}

Let $\Delta = (\Delta_+, \Delta_-, \delta_*)$ be a thick twin building of type $(W, S)$, let $\epsilon \in \{+,-\}$ and let $P \subseteq \Cm, P_0, \ldots, P_k \subseteq \Cp$ be panels such that $(P_0, \ldots, P_k)$ is a compatible path. Then we call this path \textit{$P$-compatible} if $P_0$ is opposite to $P$, and if $\proj_{R(P_{i-1}, P_i)} P = P_i$ for all $1 \leq i \leq k$. We obtain $\proj_{R(P_{i-1}, P_i)} x = \proj_{P_i} x$ for all $x\in P$ and $1 \leq i \leq k$, since $\proj_{R(P_{i-1}, P_i)} x \in P_i$. Furthermore, we obtain $\proj_{P_{i-1}} x = \proj_{P_{i-1}} \proj_{P_i} x$ for any $1 \leq i \leq k$ by Lemma \ref{residueprojectionabsorbtion}.


\begin{proposition}\label{projectiontransitive}
	Let $\epsilon \in \{+,-\}$ and $P \subseteq \Cm, P_0, \ldots, P_k \subseteq \Cp$ be panels such that $(P_0, \ldots, P_k)$ is a $P$-compatible path. Then the following hold for all $0 \leq i \leq k$:
	\begin{enumerate}[label=(\alph*), leftmargin=*]
		\item $\proj_{P_0}^P = \proj_{P_0}^{P_i} \circ \proj_{P_i}^P$.
		
		\item $\proj_P^{P_0} = \proj_P^{P_i} \circ \proj_{P_i}^{P_0}$.
		
		\item $\ell_*(x, \proj_{P_i} x) = \ell(\delta(P_0, P_i)) +1$ for each $x\in P$.
	\end{enumerate}
\end{proposition}
\begin{proof}
	It suffices to show the claim only for $i=k$, since $(P_0, \ldots, P_i)$ is a $P$-compatible path. For $k=0$ there is nothing to show. For $x\in P$ we have $\proj_{P_{k-1}} \proj_{P_k} x = \proj_{P_{k-1}} x$. Using induction and Proposition \ref{DMVMProposition9}$(a)$ and $(d)$ we obtain
	\begin{align*}
	\proj_{P_0} x = \proj_{P_0} \proj_{P_{k-1}} x = \proj_{P_0} \proj_{P_{k-1}} \proj_{P_k} x = \proj_{P_0} \proj_{P_k} x. 
	\end{align*}
	This proves Assertion $(a)$. Using Lemma \ref{parallelresidues} and Lemma \ref{AB08Exercise5.168+projbijection+Mu97Lemma3.5}$(b)$ and $(d)$, the panels $P$ and $P_i$ are parallel. Using again Lemma \ref{parallelresidues} and Lemma \ref{AB08Exercise5.168+projbijection+Mu97Lemma3.5}$(b)$ and $(d)$, Assertion $(b)$ is a consequence of Assertion $(a)$. For Assertion $(c)$ it also suffices to show the claim for $i=k$. Let $x\in P$. Since $\proj_{P_{k-1}} \proj_{P_k} x = \proj_{P_{k-1}} x$, we infer $\ell_*(x, \proj_{P_{k-1}} x) = \ell_*(x, \proj_{P_k} x) - \ell_{-\epsilon}(\proj_{P_k} x, \proj_{P_{k-1}} x) = \ell_*(x, \proj_{P_k} x) - \ell(\delta(P_{k-1}, P_k))$. Using Proposition \ref{DMVMProposition9}$(c)$ and induction we obtain $\ell_*(x, \proj_{P_k} x) = \ell_*(x, \proj_{P_{k-1}} x) + \ell(\delta(P_{k-1}, P_k)) = \ell(\delta(P_0, P_k)) +1$ and the claim follows.
\end{proof}

\begin{lemma}\label{GeneralizationDMVM11Lemma14}
	Let $\epsilon \in \{+,-\}, P_1 \subseteq \Cp, P_2 \subseteq \Cm$ be two parallel panels. Let $s_i$ be the type of $P_i$. Then $s_2 = w^{-1} s_1 w$, where $w := \delta_*(x, \proj_{P_2} x)$ does not depend on the choice of $x$ in $P_1$.
	
	Conversely, if $x$ and $y$ are chambers with $\delta_*(x, y) = w$, where $w$ satisfies $s_2 = w^{-1} s_1 w$ and $\ell(s_1w) = \ell(w) -1$, then the $s_1$-panel of $x$ is parallel to the $s_2$-panel of $y$.
\end{lemma}
\begin{proof}
	We follow the ideas of \cite[Lemma $14$]{DMVM11}. Let $x_1 \neq y_1 \in P_1$, let $x_2 := \proj_{P_2} x_1, y_2 := \proj_{P_2} y_1$ and let $w = \delta_*(x_1, x_2), w' = \delta_*(y_1, y_2)$. Then we have $\delta_*(x_2, y_1) = w^{-1}s_1$ and $\delta_*(y_1, x_2) = w's_2$. In particular, $s_1w = w's_2$ and hence $w' = s_1 w s_2$. Since $\ell(w's_2) = \ell(w') -1$, we have $\ell(s_1ws_2) = \ell(w') = \ell(w's_2) +1 = \ell(s_1w) +1 = \ell(w)$. Since $\ell(ws_2) \ell(w)-1$, the claim follows from Lemma \ref{NearlyConditionF}. For the converse let $P_1$ the $s_1$-panel of $x$ and let $P_2$ the $s_2$-panel of $y$. Let $x\in P_1$. Then $\proj_{P_2} x = y$, since $\ell(ws_2) = \ell(s_1w) = \ell(w) -1$. Choose $\proj_{P_2} x \neq p\in P_2$. Then $\delta_*(x, p) = ws_2$. By (Tw$3$) there exists $z\in P_1$ such that $\delta_*(z, p) = s_1ws_2 = w$. Since $\ell(s_1ws_2) = \ell(w) = \ell(ws_2) +1$ we have $z = \proj_{P_1} p$. It follows from $\ell(s_1w) = \ell(w) -1$ that $x = \proj_{P_1} \proj_{P_2} x$. Since $\delta_*(x, p) \neq \delta_*(z, p)$, we deduce that $x$ and $\proj_{P_1} p$ are different. The claim follows now from Lemma \ref{GeneralizationDMVMLemma13}.
\end{proof}

Let $P_1$ and $P_2$ be two parallel panels in different halves. By the previous lemma $\delta_*(x, \proj_{P_2} x)$ does not depend on the choice of $x\in P_1$. Therefore we define $\delta(P_1, P_2) := \delta_*(x, \proj_{P_2} x)$, where $x$ is a chamber in $P_1$.

\begin{theorem}\label{parallelPcomppath}
	Let $\epsilon \in \{+,-\}, P \subseteq \Cp, Q \subseteq \Cm$ be two panels. Then $P, Q$ are parallel if and only if there exists a $P$-compatible path $(Q_0, \ldots, Q_k = Q)$.
\end{theorem}
\begin{proof}
	Let $(Q_0, \ldots, Q_k = Q)$ be a $P$-compatible path. Using Proposition \ref{projectiontransitive}$(a)$ we have $\proj_{Q_0}^P = \proj_{Q_0}^Q \circ \proj_Q^P$. By Lemma \ref{AB08Exercise5.168+projbijection+Mu97Lemma3.5}$(d)$ we have $\proj_{Q_0} P = Q_0$. Thus $\vert \proj_Q P \vert \geq 2$ and Lemma \ref{GeneralizationDMVMLemma13} finishes the claim.
	
	Now we assume that $P$ and $Q$ are parallel. We show the claim via induction on the distance $\ell := \ell(\delta(P, Q))$. If $\ell=1$ then $(Q)$ is a $P$-compatible path and we are done. Now we assume that $\ell > 1$. Let $x \in P$. Then there exists a chamber $e \in \Cm$ which is adjacent to a chamber in $Q$ and satisfies $\ell_*(x, \proj_Q x) -2 = \ell_*(x, e)$. Let $R$ be the unique rank $2$ residue containing $Q$ and $e$. By Lemma \ref{GeneralizationDMVM11Lemma17} the panels $Q$ and $\proj_R P$ are parallel. Since $Q$ and $\proj_R P$ are not opposite in $R$ we have $\proj_R P = Q$ by Lemma \ref{DMVMLemma18}. Note also that $\proj_R x = \proj_Q x$. Let $w = \delta_*(x, \proj_Q x)$. Let $q \in \Cm$ such that $\delta_{-\epsilon}(\proj_Q x, q) = w^{-1}$. Then $\delta_*(x, q) = 1_W$ by \cite[Lemma $5.140(2)$]{AB08}. Let $s$ be the type of $P$ and let $t$ be the type of $Q$. Then $\ell(wt) = \ell(w) -1$. In particular, $\delta_{-\epsilon}(q, \proj_Q q) = wt$. Note that for $v := tw^{-1}$ we have $\ell(tv) = \ell(w^{-1}) = \ell(w) = \ell(wt) +1 = \ell(v) +1$. Since $P, Q$ are parallel, we have $s = wtw^{-1}$ by Lemma \ref{GeneralizationDMVM11Lemma14}. This implies $v^{-1}tv = wtttw^{-1} = s$. Using Lemma \ref{DMVMLemma14} we obtain that the $s$-panel of $q$ and the $t$-panel of $\proj_Q q$ (i.e. $Q$) are parallel. Let $Q_0$ be the $s$-panel of $q$. By \cite[Lemma $17$]{DMVM11} $\proj_R Q_0$ is parallel to both, $Q_0$ and $Q$. Since $\ell(wr_J) = \ell(w) - \ell(r_J)$, where $J$ is the type of $R$, we have $\proj_R Q_0 \neq Q$. Using induction, there exists a $P$-compatible path $(Q_0, \ldots, Q_k = \proj_R Q_0)$. Clearly, $(Q_0, \ldots, Q_k, Q)$ is a compatible path. The fact that $\proj_{R(Q_k, Q)} P = \proj_R P = Q$ finishes the claim.
\end{proof}

\begin{theorem}\label{P-compatiblepath}
	Let $\epsilon \in \{+,-\}, s\in S$ and let $P, Q \subseteq \Cp$ be two parallel panels and let $c_- \in \Cm$ such that $\P_s(c_-)$ and $P$ are opposite and such that $\ell(\delta(P, Q)) +1 = \ell_*(c_-, \proj_Q c_-)$. Then every compatible path $(P_0 = P, \ldots, P_k = Q)$ is a $\P_s(c_-)$-compatible path.
\end{theorem}
\begin{proof}
	At first we show $\ell_*(c_-, \proj_{P_i} c_-) = \ell(P, P_i) +1$. We have $\ell_*(c_-, \proj_P \proj_Q c_-) \geq \ell_*(c_-, \proj_Q c_-) - \ell_{\epsilon}(\proj_Q c_-, \proj_P \proj_Q c_-) = 1$ by hypothesis. Since the panels $\P_s(c_-), P$ are opposite, we obtain $\proj_P \proj_Q c_- = \proj_P c_-$ by Lemma \ref{Mu97Lemma3.4}. Let $c_+ \in P\backslash \{ \proj_P c_- \}$. Then $c_+ \in c_-^{\mathrm{op}}$. Since $c_+ \neq \proj_P c_- = \proj_P \proj_Q c_-$, we have $\ell_{\epsilon}(\proj_Q c_-, c_+) = \ell(\delta(P, Q)) +1$. This yields $\proj_Q c_- \in A(c_+, c_-)$, since $\ell_{\epsilon}(c_+, \proj_Q c_-) = \ell_*(c_-, \proj_Q c_-)$. Now we will show that $A(c_+, c_-) \cap P_i \neq \emptyset$. Let $z_i := \proj_{P_i} \proj_Q c_-$. Then by Proposition \ref{DMVMProposition9}$(a), (c), (d)$ the following hold:
	\begin{align*}
	\ell_{\epsilon}(\proj_Q c_-, c_+) &= \ell(\delta(P, Q)) + 1 \\
	&= \ell(\delta(P, P_i)) + \ell(\delta(P_i, Q)) + 1 \\
	&= \ell_{\epsilon}(\proj_Q c_-, z_i) + \ell_{\epsilon}(z_i, \proj_P z_i) + \ell_{\epsilon}(\proj_P \proj_Q c_-, c_+) \\
	&= \ell_{\epsilon}(\proj_Q c_-, z_i) + \ell_{\epsilon}(z_i, c_+).
	\end{align*}
	This yields that $z_i$ lies on a minimal gallery between $\proj_Q c_-$ and $c_+$. The definition of convexity implies that any element on a minimal gallery between two chambers of a convex set is contained in the convex set. Since $A_{\epsilon}(c_+, c_-)$ is convex, we infer $z_i \in A(c_+, c_-) \cap P_i \neq \emptyset$. By Lemma \ref{Twinapartment} we obtain $A(c_+, c_-) \cap P_i = \{ \proj_{P_i} c_+, \proj_{P_i} c_- \}$. We put $c_i := \proj_{P_i} c_-$. Then the following hold for every $0 \leq i \leq k$:
	\begin{align*}
	\ell_*(c_-, c_i) = \ell_{\epsilon}(c_+, c_i) = \ell_{\epsilon}(c_+, \proj_{P_i} c_+) + \ell_{\epsilon}(\proj_{P_i} c_+, c_i) = \ell(\delta(P, P_i)) + 1.
	\end{align*}
	Now we want to show that the panels $P_i$ and $\P_s(c_-)$ are parallel. Let $P_i$ be a $t$-panel and let $w' = \delta(P, P_i)$. Then we have $t = w'^{-1} sw'$ by Lemma \ref{DMVMLemma14}. Let $w = \delta_*(c_-, \proj_{P_i} c_-)$ and let $\proj_P c_- \neq c_+\in P$. Note that $\proj_P c_- = \proj_P \proj_{P_i} c_-$ for any $0 \leq i \leq k$ as above. Then $w = \delta_*(c_-, \proj_{P_i} c_-) = \delta_{\epsilon}(c_+, \proj_{P_i} c_-) = w't = sw'$ by the definition of $A(c_+, c_-)$. We have $w^{-1} s w = (sw')^{-1} s (sw') = t$, $\ell(sw) = \ell(w') = \ell(sw') -1 = \ell(w) -1$ and hence Lemma \ref{GeneralizationDMVM11Lemma14} implies that the panels $\P_s(c_-)$ and $P_i$ are parallel. In particular, we have $\ell(x, \proj_{P_i} x) = \ell(c_-, \proj_{P_i} c_-)$ by Lemma \ref{GeneralizationDMVM11Lemma14}.
	
	We now show the claim. Since $\P_s(c_-)$ and $P_i$ are parallel, it suffices to show that $\proj_{R(P_{i-1}, P_i)} x = \proj_{P_i} x$ for all $1 \leq i \leq k$ and $x\in \P_s(c_-)$. Let $1 \leq i \leq k$, let $\proj_{P_{i-1}} x \neq y \in P_{i-1}$ and let $J_i \subseteq S$ be the type of $R(P_{i-1}, P_i)$. Let $r_i := \proj_{R(P_{i-1}, P_i)} x, p_{i-1} := \proj_{P_{i-1}} x$ and $p_i := \proj_{P_i} x$. Then we obtain:
	\begin{align*}
	\ell_*(x, r_i) - \ell(r_{J_i}) &\leq \ell_*(x, y) = \ell_*(x, p_{i-1}) - \ell_{\epsilon}(p_{i-1}, y) = \ell_*(x, r_i) - \ell_{\epsilon}(r_i, p_{i-1}) - 1
	\end{align*}
	Therefore, we have $\ell_{\epsilon}(r_i, p_{i-1}) \leq \ell(r_{J_i}) -1 = \ell(\delta(P_{i-1}, P_i))$. Using Proposition \ref{DMVMProposition9}$(c)$, we deduce:
	\begin{align*}
	\ell_*(x, r_i) - \ell_{\epsilon}(r_i, p_i) &= \ell_*(x, p_i) \\
	&= \ell(\delta(P, P_i)) +1 \\
	&= \ell(\delta(P, P_{i-1})) + \ell(\delta(P_{i-1}, P_i)) +1 \\
	&= \ell_*(x, p_{i-1}) + \ell(\delta(P_{i-1}, P_i)) \\
	&= \ell_*(x, r_i) - \ell_{\epsilon}(r_i, p_{i-1}) + \ell(\delta(P_{i-1}, P_i)) \\
	&\geq \ell_*(x, r_i).
	\end{align*}
	This implies $\ell_{\epsilon}(r_i, p_i) = 0$ and the claim follows.
\end{proof}

\begin{definition}\label{Definition: wall-connected}
	Let $\epsilon \in \{+,-\}, s \in S, c \in \Cm$. Two panels $Q_1, Q_2 \subseteq \Cp$ are called \textit{wall-adjacent of type $(c, s)$} if both panels $Q_1, Q_2$ are opposite to $\P_s(c)$ and if there exist a panel $T\subseteq \Cp$ and $\P_s(c)$-compatible paths $(P_0 = Q_1, \ldots, P_k = T)$ and $(P_0' = Q_2, \ldots, P_{k'}' = T)$ of the same length and type.
	
	For any $\epsilon \in \{+, -\}$ and any pair $(c, s) \in \C_{\epsilon} \times S$ we define a graph $\Gamma_s(c)$ with vertex set $\{ Q \subseteq \C_{-\epsilon} \mid \P_s(c), Q \text{ are opposite panels} \}$ and two vertices $Q_1, Q_2$ are joined by an edge if $Q_1, Q_2$ are wall-adjacent of type $(c, s)$. The twin building $\Delta$ is called \textit{wall-connected} if the graph $\Gamma_s(c)$ is connected for each pair $(c, s) \in \C \times S$. We refer to Lemma \ref{parallelpanelssamewall} for the motivation of the terminology \textit{wall-connected}.
\end{definition}

\section{Isometries}

\subsection*{Definition and basic facts}

\begin{convention}
	From now on all buildings (and hence all twin buildings) are assumed to be thick.
\end{convention}

Let $\Delta = (\Delta_+, \Delta_-, \delta_*), \Delta' = (\Delta_+', \Delta_-', \delta_*')$ be two twin buildings of type $(W, S)$. We define $\C', \Delta_+', \Delta_-', \delta', \ell'$ as in the case of $\Delta$. Let $\mathcal{X} \subseteq \C, \mathcal{X}' \subseteq \C'$. A mapping $\phi: \mathcal{X} \to \mathcal{X}'$ is called \textit{isometry} if the following conditions are satisfied:
\begin{enumerate}[label=(Iso\arabic*), leftmargin=*]
	\item The mapping $\phi$ is bijective.
	
	\item For $\epsilon \in \{+,-\}$ we have $\phi(\mathcal{X} \cap \Cp) \subseteq \Cp'$.
	
	\item If $x, y \in \mathcal{X}$ then $\delta'(\phi(x), \phi(y)) = \delta(x, y)$.
\end{enumerate}

It is easy to see that $\phi^{-1}$ is also an isometry. Given $\mathcal{X} \subseteq \C, \mathcal{X}' \subseteq \C'$, an isometry $\phi: \mathcal{X} \to \mathcal{X}'$ and $(y, y') \in \C \times \C'$, then the pair $(y, y')$ will be called \textit{$\phi$-admissible} if the mapping $y \mapsto y'$ extends $\phi$ to an isometry from $\mathcal{X} \cup \{y\}$ onto $\mathcal{X}' \cup \{y'\}$. Let $(x, x')$ be a $\phi$-admissible pair. Then $(x', x)$ is $\phi^{-1}$-admissible. In particular, $(x, \phi(x))$ is $\phi$-admissible for any $x\in \mathcal{X}$. Let $\mathcal{Y} \subseteq \C, \mathcal{Y}' \subseteq \C'$ and $\psi: \mathcal{Y} \to \mathcal{Y}'$ be another isometry. Then the pair $(\phi, \psi)$ will be called \textit{admissible}, if there exists an isometry from $\mathcal{X} \cup \mathcal{Y}$ onto $\mathcal{X}' \cup \mathcal{Y}'$ such that $\phi$ and $\psi$ are restrictions of that isometry.

\begin{lemma}\label{adjazenzerhaltendeBijektionIsometrie}
	Let $\epsilon \in \{+,-\}$ and $\phi: \Cp \to \Cp'$ be a bijection. If $\delta_{\epsilon}(x, y) = \delta_{\epsilon}'(\phi(x), \phi(y))$ for any $x, y \in \Cp$ with $\delta_{\epsilon}(x, y) \in S$, then $\phi$ is an isometry.
\end{lemma}
\begin{proof}
	This is \cite[Lemma $5.61$]{AB08}.
\end{proof}

\begin{lemma}\label{ZusammengesetzteIsometrie}
	Let $\mathcal{X}, \mathcal{Y} \subseteq \C, \mathcal{X}', \mathcal{Y}' \subseteq \C'$ be such that $\mathcal{X} \cap \mathcal{Y} = \emptyset$ and $\mathcal{X}' \cap \mathcal{Y}' = \emptyset$. Let $\phi: \mathcal{X} \to \mathcal{X}'$ and $\psi: \mathcal{Y} \to \mathcal{Y}'$ be two isometries such that $(z, \psi(z))$ is $\phi$-admissible for any $z\in \mathcal{Y}$. Then $(\phi, \psi)$ is admissible.
\end{lemma}
\begin{proof}
	This is a consequence of \cite[Lemma $4.1$]{BCM21}.
\end{proof}

\begin{lemma}\label{Mu97Lemma4.2}
	Let $J$ be a spherical subset of $S$, let $R \subseteq \C, R' \subseteq \C'$ be $J$-residues, let $\phi: R \to R'$ be an isometry, and let $(x, x')$ be a $\phi$-admissible pair. Then $\phi(\proj_R x) = \proj_{R'} x'$.
\end{lemma}
\begin{proof}
	This is \cite[Lemma $(4.4)$]{Ro00}.
\end{proof}

\begin{lemma}\label{Mu97Lemma4.3}
	Let $R_+, R_- \subseteq \C$ be spherical and parallel residues in $\Delta$ and $R_+', R_-' \subseteq \C'$ be spherical and parallel residues in $\Delta'$, let $\phi: R_+ \cup R_- \to R_+' \cup R_-'$ be an isometry such that $\phi(R_+) = R_+'$ and $\phi(R_-) = R_-'$. Then $\phi(x) = \proj_{R_{\epsilon}'} \phi(\proj_{R_{-\epsilon}} x)$ for each $x\in R_{\epsilon}$ for each $\epsilon \in \{+,-\}$.
\end{lemma}
\begin{proof}
	This is a consequence of the previous lemma, Lemma \ref{parallelresidues} and Lemma \ref{AB08Exercise5.168+projbijection+Mu97Lemma3.5}$(b)$.
\end{proof}

\begin{lemma}\label{Mu97Lemma4.6}
	Let $\phi_+: \C_+ \to \C_+'$ be an isometry, let $(x, x') \in \C_- \times \C_-'$ such that $\phi_+(x^{\mathrm{op}}) \subseteq (x')^{\mathrm{op}}$. Then $(x, x')$ is a $\phi_+$-admissible pair.
\end{lemma}
\begin{proof}
	This is \cite[Lemma $(7.4)$]{Ro00}.
\end{proof}

\begin{lemma}\label{MRLemma4.10}
	Let $x_- \in \C_-$, $x_-', y_-' \in \C_-'$ and let $\phi: \C_+ \cup \{x_-\} \to \C_+' \cup \{x_-'\}, \psi: \C_+ \cup \{x_-\} \to \C_+' \cup \{y_-'\}$ be two isometries such that $\phi(z) = \psi(z)$ for any $z\in x_-^{\mathrm{op}}$. Then $x_-' = y_-'$ and $\phi = \psi$.
\end{lemma}
\begin{proof}
	This is \cite[Lemma $4.10$]{MR95}.
\end{proof}

\begin{theorem}\label{Ro00Theorem1}
	Let $(c_+, c_-)$ be a pair of opposite chambers in $\Delta$. The only isometry of $\Delta$ fixing $E_1(c_+) \cup \{c_-\}$ is the identity.
\end{theorem}
\begin{proof}
	This is \cite[Theorem $(3.2)$]{Ro00}.
\end{proof}

\begin{theorem}\label{Mu97Theorem1}
	Let $\Delta, \Delta'$ be $2$-spherical and of rank at least three. Let $(c_+, c_-), (c_+', c_-')$ be two pairs of opposite chambers in $\Delta$ and $\Delta'$, respectively, and let $\phi: E_2(c_+) \cup \{c_-\} \to E_2(c_+') \to \{c_-'\}$ be an isometry. Then $\phi$ extends to an isometry from $\C_+ \cup \{c_-\}$ onto $\C_+' \cup \{c_-'\}$.
\end{theorem}
\begin{proof}
	This is a consequence of \cite[Theorem $6.5$]{BCM21}.
\end{proof}

\subsection*{Isometries and wall-connected twin buildings}

Let $\Delta, \Delta'$ be two twin buildings of type $(W, S)$. Let $(c_+, c_-) \in \C_+ \times \C_-, (c_+', c_-') \in \C_+' \times \C_-'$ be pairs of opposite chambers and let $\phi_+: \C_+ \to \C_+'$ be an isometry such that $\phi_+(c_+) = c_+'$. Furthermore let $(c_-, c_-')$ be $\phi_+$-admissible.

\begin{lemma}\label{Lemma:isometrymapsPcompatiblepathontoPcompatiblepath}
	Let $(P_0, \ldots, P_k)$ be an $\P_s(c_-)$-compatible path. Then $(\phi_+(P_0), \ldots, \phi_+(P_k))$ is an $\P_s(c_-')$-compatible path.
\end{lemma}
\begin{proof}
	Clearly, $(\phi_+(P_0), \ldots, \phi_+(P_k))$ is a compatible path by Lemma \ref{Mu97Lemma4.2}. Now $\phi_+(P_0)$ and $\P_s(\phi_+(c_-)) = \P_s(c_-')$ are opposite and by Proposition \ref{projectiontransitive}$(c)$ and Lemma \ref{Mu97Lemma4.2} we have $\ell(c_-', \proj_{\phi_+(P_k)} c_-') = \ell(\phi_+(c_-), \phi_+(\proj_{P_k} c_-)) = \ell(c_-, \proj_{P_k} c_-) = \ell(P_0, P_k) +1 = \ell(\phi_+(P_0), \phi_+(P_k)) +1$. Now the claim follows from Theorem \ref{P-compatiblepath}.
\end{proof}

\begin{lemma}\label{Lemma:Wall-connected survives under isometries}
	Let $\Delta, \Delta'$ be $2$-spherical. Then $\Gamma_s(c_-)$ is connected if and only if $\Gamma_s(c_-')$ is connected.
\end{lemma}
\begin{proof}
	It suffices to show that if $Q_1, Q_2$ are wall-adjacent of type $(c_-, s)$, then $\phi_+(Q_1), \phi_+(Q_2)$ are wall-adjacent of type $(c_-', s)$, too. Let $Q_1, Q_2$ be wall-adjacent of type $(c_-, s)$. By definition $Q_1, Q_2 \in \P_s(c_-)^{\mathrm{op}}$ and there exist a panel $T \subseteq \C_+$ and $\P_s(c_-)$-compatible paths $(P_0 = Q_1, \ldots, P_k = T)$ and $(P_0' = Q_2, \ldots, P_k' = T)$ of the same length and type. Lemma \ref{Lemma:isometrymapsPcompatiblepathontoPcompatiblepath} implies that $\phi_+$ maps a $\P_s(c_-)$-compatible to a $\P_s(c_-')$-compatible path. Thus $\phi_+(Q_1), \phi_+(Q_2)$ are wall-adjacent of type $(c_-', s)$ and the claim follows.
\end{proof}

\begin{definition}
	For $x \in c_-^{\mathrm{op}}$ and $s\in S$ we define the mapping
	\begin{align*}
	\phi_s^x: \P_s(c_-) \to \P_s(c_-'), d \mapsto \left( \proj_{\P_s(c_-')}^{\P_s(\phi_+(x))} \circ \phi_+ \circ \proj_{\P_s(x)}^{\P_s(c_-)} \right) (d).
	\end{align*}
	Since $\P_s(c_-), \P_s(x)$ are parallel by Lemma \ref{AB08Exercise5.168+projbijection+Mu97Lemma3.5}$(d)$ for each $x\in c_-^{\mathrm{op}}$ and $\phi_+$ is an isometry it follows again by Lemma \ref{AB08Exercise5.168+projbijection+Mu97Lemma3.5}$(d)$ that $\phi_s^x$ is a bijection and hence an isometry. In particular, $\phi_s^x(c_-) = c_-'$.
\end{definition}

\begin{proposition}\label{Propphiswelldefined}
	Let $s\in S, x, z \in c_-^{\mathrm{op}}$ and let $P = \P_s(x), Q = \P_s(z)$ be wall-adjacent of type $(c_-, s)$. Then we have $\phi_s^x = \phi_s^z$.
\end{proposition}
\begin{proof}
	By definition there exist a panel $T \subseteq \C_+$ and $\P_s(c_-)$-compatible path's $(P_0 = P, \ldots, P_k = T)$ and $(Q_0 = Q, \ldots, Q_k = T)$ of the same length and type. By the previous lemma $(\phi_+(P_0), \ldots, \phi_+(P_k))$ and $(\phi_+(Q_0), \ldots, \phi_+(Q_k))$ are $\P_s(c_-')$-compatible paths. Let $Z \in \{P, Q\}$. By Proposition \ref{projectiontransitive}$(a), (b)$ we obtain
	\begin{align*}
		\proj_{Z}^{\P_s(c_-)} &= \proj_{Z}^{T} \circ \proj_{T}^{\P_s(c_-)} \\
		\proj_{\P_s(c_-')}^{\phi_+(Z)} &= \proj_{\P_s(c_-')}^{\phi_+(T)} \circ \proj_{\phi_+(T)}^{\phi_+(Z)} 
	\end{align*}
	By Lemma \ref{Mu97Lemma4.3} we obtain $\proj_{\phi_+(T)}^{\phi_+(Z)} \circ \phi_+ \circ \proj_{Z}^T = \phi_+ \vert_T$, since the panels $Z$ and $T$ are parallel by Lemma \ref{DMVMLemma19}. We have
	\begin{align*}
		\proj_{\P_s(c_-')}^{\phi_+(Z)} \circ \phi_+ \circ \proj_{Z}^{\P_s(c_-)} &= 
		\proj_{\P_s(c_-')}^{\phi_+(T)} \circ \proj_{\phi_+(T)}^{\phi_+(Z)} \circ \phi_+ \circ \proj_{Z}^{T} \circ \proj_{T}^{\P_s(c_-)} \\
		&= \proj_{\P_s(c_-')}^{\phi_+(T)} \circ  \phi_+ \circ \proj_{T}^{\P_s(c_-)}
	\end{align*}
	This finishes the claim.
\end{proof}

\begin{corollary}\label{phiswelldefined}
	Let $x, z \in c_-^{\mathrm{op}}$ and $s\in S$ such that $\Gamma_s(c_-)$ is connected. Then $\phi_s^x = \phi_s^z$.
\end{corollary}
\begin{proof}
	This follows by induction on the length of a path in $\Gamma_s(c_-)$ and Proposition \ref{Propphiswelldefined}.
\end{proof}

\subsection*{Extending isometries of wall-connected twin buildings}

Let $\Delta, \Delta'$ be two twin buildings of type $(W, S)$. Let $(c_+, c_-) \in \C_+ \times \C_-, (c_+', c_-') \in \C_+' \times \C_-'$ be pairs of opposite chambers and let $\phi_+: \C_+ \to \C_+'$ be an isometry such that $\phi_+(c_+) = c_+'$. Furthermore let $(c_-, c_-')$ be $\phi_+$-admissible. Assume that $\Delta$ is wall-connected. By Corollary \ref{phiswelldefined} we have $\phi_s^x = \phi_s^z$ for any $x, z \in c_-^{\mathrm{op}}$ and $s\in S$. We denote this mapping by $\phi_s$.

\begin{lemma}\label{panelextension}
	For each $s\in S$ the pair $(\phi_+, \phi_s)$ is admissible.
\end{lemma}
\begin{proof}
	Let $d_-\in \P_s(c_-)$. We will show that $\phi_+(d_-^{\mathrm{op}}) \subseteq \phi_s(d_-)^{\mathrm{op}}$. Let $y\in d_-^{\mathrm{op}}$. By Lemma \ref{Lemadjacent} there exists $x \in c_-^{\mathrm{op}}$ such that $\delta_+(y, x) = s$. Since $y \in d_-^{\mathrm{op}}$ we have $\proj_{\P_s(x)} d_- \neq y$. This implies $s = \delta_+(\proj_{\P_s(x)} d_-, y) = \delta_+'(\phi_+(\proj_{\P_s(x)} d_-), \phi_+(y))$. Since $\P_s(c_-')$ and $\P_s(\phi_+(x))$ are opposite, Lemma \ref{Mu97Lemma3.4} and the definition of $\phi_s$ yield $\delta_*'( \phi_s(d_-), \phi_+ (\proj_{\P_s(x)} d_-)) = s$ by Lemma \ref{Mu97Lemma3.4}. By (Tw$2$) we have $\delta_*'( \phi_s(d_-), \phi_+(y)) = 1_W$. By Lemma \ref{Mu97Lemma4.6} the pair $(d_-, \phi_s(d_-))$ is $\phi_+$-admissible for all $d_-\in \P_s(c_-)$. The claim follows now by Lemma \ref{ZusammengesetzteIsometrie}.
\end{proof}

\begin{lemma}
	The isometry $\phi_+$ extends uniquely to an isometry $\phi: \C_+ \cup E_1(c_-) \to \C_+' \cup E_1(c_-')$. In particular, for every chamber $x\in \C_-$ there exists a unique chamber $x' \in \C_-'$ such that $(x, x')$ is $\phi_+$-admissible.
\end{lemma}
\begin{proof}
	Let $s, t \in S$. Then $\phi_s(c_-) = c_-' = \phi_t(c_-)$. Therefore the mapping $\phi_-: E_1(c_-) \to E_1(c_-'), x \mapsto \phi_s(x)$, if $x\in \P_s(c_-)$, is well-defined. Moreover, $\phi_-$ is bijective and for all $x, y \in E_1(c_-)$ with $\delta_-(x, y) \in S$ we have $\delta_-(x, y) = \delta_-'(\phi_-(x), \phi_-(y))$. Let $x, y \in E_1(c_-)$ with $\delta_-(x, y) \notin S$. Then there exists $s \neq t \in S$ with $x \in \P_s(c_-)$ and $y\in \P_t(c_-)$ and we have $\delta_-(x, y) = \delta_-(x, c_-) \delta_-(c_-, y) = st$. We deduce $\delta_-'(\phi_-(x), \phi_-(c_-)) = s$ and $\delta_-'(\phi_-(c_-), \phi_-(y)) = t$ and hence $\delta_-'(\phi_-(x), \phi_-(y)) = st$. Thus $\phi_-$ is an isometry and we obtain that $(x, \phi_-(x))$ is $\phi_+$-admissible for all $x\in E_1(c_-)$ by Lemma \ref{panelextension}. By Lemma \ref{ZusammengesetzteIsometrie} the pair $(\phi_+, \psi_-)$ is admissible and hence the mapping $\phi_-$ extends $\phi_+$ to an isometry from $\C_+ \cup E_1(c_-)$ to $\C_+' \cup E_1(c_-')$.
	
	The uniqueness of $x'$ follows from Lemma \ref{MRLemma4.10}. The existence follows by the first assertion of the lemma and induction on $\ell_-(c_-, x)$.
\end{proof}

\begin{theorem}\label{ExtendingIsometry}
	Let $x\in \C_-$ and let $x' \in \C_-'$ be the unique chamber such that $(x, x')$ is $\phi_+$-admissible. Then $\phi_-: \C_- \to \C_-', x \mapsto x'$ is an isometry.
\end{theorem}
\begin{proof}
	Let $(y, x')$ be $\phi_+$-admissible. Then $(x', y)$ and $(x', x)$ are $\phi_+^{-1}$-admissible. The uniqueness part of the previous lemma implies $x=y$ and hence $\phi_-$ is injective. Let $y' \in \C_-'$ and let $y \in \C_-$ be the unique chamber such that $(y', y)$ is $\phi_+^{-1}$-admissible. Then $(y, y')$ is $\phi_+$-admissible and hence $\phi_-$ is surjective. Thus $\phi_-$ is a bijection.
	Now we will show that $\phi_-$ preserves $s$-adjacency. Let $x, y \in \C_-, s\in S$ such that $\delta_-(x, y) = s$. Note that $\phi_+$ is an isometry and that $(x, \phi_-(x)), (y, \phi_-(y))$ are $\phi_+$-admissible. Let $z \in \phi_+(x)^{\mathrm{op}}$. Then $\phi_+^{-1}(z) \in x^{\mathrm{op}}$ and Lemma \ref{Lemadjacent} yields $z' \in y^{\mathrm{op}}$ such that $\delta_+(z, \phi_+(z')) = \delta_+(\phi_+^{-1}(z), z') = s$. Again, Lemma \ref{Lemadjacent} implies $\delta_-'(\phi_-(x), \phi_-(y)) \in \langle s \rangle$. Since $\phi_-$ is injective, we have $\delta_-'(\phi_-(x), \phi_-(y)) = s$. Now Lemma \ref{adjazenzerhaltendeBijektionIsometrie} finishes the claim.
\end{proof}

\begin{corollary}\label{Cor: Main result}
	Let $\Delta$ be $2$-spherical, thick, wall-connected and of rank at least three. Then any isometry $\phi: E_2(c_+) \cup \{c_-\} \to E_2(c_+') \cup \{c_-'\}$ extends uniquely to an isometry from $\C_+ \cup \C_-$ onto $\C_+' \cup \C_-'$.
\end{corollary}
\begin{proof}
	By Theorem \ref{Mu97Theorem1}, Theorem \ref{ExtendingIsometry} and Lemma \ref{ZusammengesetzteIsometrie} we obtain an isometry $\Phi: \C \to \C'$ such that $\Phi \vert_{E_2(c_+) \cup \{c_-\}} = \phi$. The uniqueness follows from Theorem \ref{Ro00Theorem1}.
\end{proof}

\section{Wall-connected twin buildings}\label{Section: Wall-connected twin buildings}

\subsection*{Chamber systems}

Let $I$ be a set. A \textit{chamber system} over $I$ is a pair $\Cbf = \left( \C, (\sim_i)_{i\in I} \right)$ where $\C$ is a non-empty set whose elements are called \textit{chambers} and where $\sim_i$ is an equivalence relation on the set of chambers for each $i\in I$. Given $i\in I$ and $c, d \in \C$, then $c$ is called \textit{$i$-adjacent} to $d$ if $c \sim_i d$. The chambers $c, d$ are called \textit{adjacent} if they are $i$-adjacent for some $i \in I$.

A \textit{gallery} in $\Cbf$ is a sequence $(c_0, \ldots, c_k)$ such that $c_{\mu} \in \C$ for all $0 \leq \mu \leq k$ and such that $c_{\mu -1}$ is adjacent to $c_{\mu}$ for all $1 \leq \mu \leq k$. The chamber system $\Cbf$ is said to be \textit{connected}, if for any two chambers $c, d$ there exists a gallery $(c_0 = c, \ldots, c_k = d)$. For a subset $\mathcal{E} \subseteq \C$ the restriction $(\mathcal{E}, (\sim_i \vert_{\mathcal{E} \times \mathcal{E}})_{i\in I})$ is again a chamber system over $I$.

Let $\Delta = (\C, \delta)$ be a building of type $(W, S)$. Then we define the chamber system $\Cbf(\Delta)$ as follows: The set of chambers is given by the set of chambers $\C$ of $\Delta$ and two chambers $x, y$ are defined to be $s$-adjacent if $\delta(x, y) \in \langle s \rangle$.

\subsection*{$3$-spherical twin buildings}

Let $\Delta = (\Delta_+, \Delta_-, \delta_*)$ be a twin building of type $(W, S)$ and let $\epsilon \in \{+,-\}$. For each pair $(c, k) \in \Cp \times \NNN$ we put $c^{\mathrm{op}(k)} := \{ d\in \Cm \mid \ell_*(c, d) \leq k \}$. We remark that $c^{\mathrm{op}} = c^{\mathrm{op}(0)}$ for any $c\in \C$. We say that the twin building $\Delta$ satisfies Condition $\co_k$, if for any $\epsilon \in \{+, -\}$ and any chamber $c\in \Cp$ the chamber system given by the restriction of $\Cbf(\Delta_{-\epsilon})$ to $c^{\mathrm{op}(k)}$ is connected. We say for short that $\Delta$ satisfies Condition $\co$ if it satisfies Condition $\co_0$.

For buildings $\Delta = (\C, \delta)$ of spherical type $(W, S)$ we have also a notion of $c^{\mathrm{op}(k)}$, i.e. $c^{\mathrm{op}(k)} = \{ d\in \C \mid \ell(c, d) \geq \ell(r_S) -k \}$. We say that a spherical building satisfies Condition $\co_k$ if $c^{\mathrm{op}(k)}$ is connected for any chamber $c\in \C$.

\begin{proposition}\label{3sphco1}
	Let $(W, S)$ be a spherical Coxeter system of rank $3$ such that $m_{st} \leq 4$ for all $s, t \in S$ and let $\Delta = (\Delta_+, \Delta_-, \delta_*)$ be a thick twin building of type $(W, S)$. Then $\Delta_+, \Delta_-$ and $\Delta$ satisfy Condition $\co_1$.
\end{proposition}
\begin{proof}
	If $m_{st} \leq 3$ for all $s, t \in S$ the assertion follows from \cite[Lemma $6.1$ and Theorem $1.5$]{MR95}. Suppose now that $m_{st} = 4$ for some $s, t \in S$. Again by \cite{MR95} the assertion holds if the $\{s, t\}$-residues are not isomorphic to the building associated with the group $C_2(2)$. In all the cases we mentioned so far, $\Delta_+, \Delta_-, \Delta_*$ even satisfies the stronger condition $\co$. If the $\{s, t\}$-residue is isomorphic to the building associated to the group $C_2(2)$, then the verification of Condition $\co_1$ boils down to an elementary calculation.
\end{proof}

A \textit{twin residue} is a pair $(R, T)$ of opposite residues in a twin building. It is a basic fact that a twin residue is again a twin building. The \textit{type} (resp. \textit{rank}) of a twin residue is defined to be the type (resp. rank) of the residues. Note that if $P$ is a panel contained in $R$ and if $(P_0, \ldots, P_k)$ is a $P$-compatible path in the twin residue $(R, T)$, then it is also a $P$-compatible path in the twin building $\Delta$. In particular, if $c\in R$ and if $s \in S$ is contained in the type of $(R, T)$ and if $Q_1, Q_2 \subseteq T$ are wall-adjacent of type $(c, s)$ in $(R, T)$, then $Q_1, Q_2$ are wall-adjacent of type $(c, s)$ in $\Delta$.

\begin{corollary}\label{Corollary: 3sph co1}
	Any $3$-spherical thick twin building $\Delta$ satisfies Condition $\co_1$.
\end{corollary}
\begin{proof}
	At first we convince ourselves that any rank $3$ residue (which is spherical by definition) satisfies $\co_1$. Let $R$ be a $J$-residue of rank $3$ and $x\in R$. For $y\in x^{\mathrm{op}}$ the residues $R_J(y), R$ are opposite in $\Delta$, i.e. $(R_J(y), R)$ is a thick spherical twin building. Hence the previous proposition implies that $R$ satisfies Condition $\co_1$.
	
	The proof is similar to the proof of \cite[Theorem $1.5$]{MR95}. Let $c$ be a chamber of $\Delta$ and let $x \neq y \in c^{\mathrm{op}(1)}$. Let $G = (x = c_0, \ldots, c_k = y)$ be a gallery. We can assume that $c_i \neq c_{i+1}$. Let $i$ be minimal such that $\ell(c, c_i) = \max\{ \ell(c, c_j) \mid 0 \leq j \leq k \}$. If $\ell(c, c_i) \leq 1$, we are done. Thus we can assume $\ell(c, c_i) >1$. Then $\ell(c, c_{i-1}) < \ell(c, c_i) \geq \ell(c, c_{i+1})$. Let $\delta(c_{i-1}, c_i) = s$ and $\delta(c_i, c_{i+1}) = t$ ($s=t$ is possible). As $\ell(c, c_{i-1}) \geq 1$, there exists $r\in S$ such that $\ell(\delta(c, c_{i-1})r) = \ell(c, c_{i-1}) -1$. Let $R$ be a $J$-residue containing $c_i$, where $\vert J \vert = 3$ and $\{r, s, t\} \subseteq J$. Using similar arguments as in \cite[Lemma $6.1$ and Theorem $1.5$]{MR95} we obtain a gallery $(c_0, \ldots, c_{i-1} = d_0, \ldots, d_m = c_{i+1}, \ldots, c_k)$ with $\ell(c, d_j) < \ell(c, c_i)$ for any $0 \leq j \leq m-1$. Iterating this procedure we get a gallery from $x$ to $y$ which is contained in $c^{\mathrm{op}(1)}$.
\end{proof}

\begin{theorem}\label{Theorem: rank 3 Condition wc}
	Let $\Delta$ be a $2$-spherical, thick twin building of type $(W, S)$ satisfying Condition $\co_1$. If any rank $3$ twin residue is wall-connected, then $\Delta$ is wall-connected.
\end{theorem}
\begin{proof}
	Let $\epsilon \in \{+,-\}, c\in \Cp$ and $s\in S$. We have to show that $\Gamma_s(c)$ is connected. Let $x, y \in c^{\mathrm{op}}$. By assumption there exists a gallery $(c_0 = x, \ldots, c_k = y)$ such that $\ell_*(c, c_i) \leq 1$ for all $0 \leq i \leq k$. Let $J = \{ \delta_{-\epsilon}(c_0, c_1), \delta_{-\epsilon}(c_1, c_2) \}$. Let $x' \in R_J(c_0) \cap E_1(c_2)$ be opposite to $c$ and let $J' = J \cup \{s\}$. Then $\vert J' \vert \leq 3$. Let $K \subseteq S$ with $\vert K \vert = 3$ and $J' \subseteq K$. By assumption the twin residue $(R_K(c), R_K(c_0))$ is wall-connected. Thus there exist $P_0 = \P_s(x), \ldots, P_m = \P_s(x')$ such that $P_{i-1}, P_i$ are wall-adjacent of type $(c, s)$ for all $1 \leq i \leq m$. Applying induction to the shorter gallery $(x', c_2, \ldots, c_k)$ the claim follows.
\end{proof}

\begin{corollary}\label{3sphwc}
	Every $3$-spherical, thick twin building is wall-connected.
\end{corollary}
\begin{proof}
	Let $\Delta$ be a $3$-spherical thick twin building. By Corollary \ref{Corollary: 3sph co1} $\Delta$ satisfies Condition $\co_1$. Let $(R, Z)$ be a twin residue of $\Delta$ of type $J$ and rank $3$. Let $c\in R$ and $s\in J$. If $(R, Z)$ is wall-connected, the claim follows from the previous theorem. Thus let $Q_1, Q_2 \subseteq Z$ such that $Q_1, Q_2 \in \P_s(c)^{\mathrm{op}}$. Then $T := \proj_Z \P_s(c)$ is a panel which is parallel to $\P_s(c), Q_1, Q_2$ by Lemma \ref{AB08Exercise5.168+projbijection+Mu97Lemma3.5} and Lemma \ref{GeneralizationDMVM11Lemma17}. Let $w := \delta(Q_1, T) = \delta(Q_2, T)$. Then $\ell(sw) = \ell(w) +1$, since $Q_1, Q_2$ are $s$-panels. Let $t$ be the type of $T$. By Lemma \ref{DMVMLemma14} we have $t = w^{-1}sw$. By Lemma \ref{DMVMLemma19} there exists a compatible path $(P_0 = Q_1, \ldots, P_k = T)$. Using Lemma \ref{DMVM11Lemma26} there exists a compatible path $(P_0' = Q_2, \ldots, P_k' = T)$ of the same length and type. Since $\ell(\delta(Q_1, T)) +1 = \ell(\delta(Q_2, T)) +1 = \ell_*(c, \proj_T c)$, both compatible paths are $\P_s(c)$-compatible by Theorem \ref{P-compatiblepath} and $(R, Z)$ is wall-connected.
\end{proof}

\begin{corollary}\label{coimplieswc}
	Every $2$-spherical, thick twin building which satisfies Condition $\co$ is wall-connected.
\end{corollary}
\begin{proof}
	Using similar arguments as in Theorem \ref{Theorem: rank 3 Condition wc} and Corollary \ref{3sphwc} the claim follows.
\end{proof}

\subsection*{Wall-connected RGD-systems}

A \textit{reflection} is an element of $W$ that is conjugate to an element of $S$. For $s\in S$ we let $\alpha_s := \{ w\in W \mid \ell(sw) > \ell(w) \}$ be the \textit{simple root} corresponding to $s$. A \textit{root} is a subset $\alpha \subseteq W$ such that $\alpha = v\alpha_s$ for some $v\in W$ and $s\in S$. We denote the set of all roots by $\Phi$. A root $\alpha \in \Phi$ is called \textit{positive} (resp. \textit{negative}), if $1_W \in \alpha$ (resp. $1_W \notin \alpha$). We let $\Phi_+$ (resp. $\Phi_-$) be the set of all positive (resp. negative) roots. For each root $\alpha \in \Phi$ we denote the \textit{opposite root} by $-\alpha$ and we denote the unique reflection which interchanges these two roots by $r_{\alpha}$. Two roots $\alpha \neq \beta \in \Phi$ are called \textit{prenilpotent} (or $\{\alpha, \beta\}$ is called a \textit{prenilpotent pair}) if $\alpha \cap \beta \neq \emptyset \neq (-\alpha) \cap (-\beta)$. For a prenilpotent pair $\{ \alpha, \beta \}$ we define $[\alpha, \beta] := \{ \gamma \in \Phi \mid \alpha \cap \beta \subseteq \gamma \text{ and } (-\alpha) \cap (-\beta) \subseteq -\gamma \}$ and $(\alpha, \beta) := [\alpha, \beta] \backslash \{ \alpha, \beta \}$.

An \textit{RGD-system of type $(W, S)$} is a pair $\mathcal{D} = \left( G, \left( U_{\alpha} \right)_{\alpha \in \Phi}\right)$ consisting of a group $G$ together with a family of subgroups $U_{\alpha}$ (called \textit{root groups}) indexed by the set of roots $\Phi$, which satisfies the following axioms, where $H := \bigcap_{\alpha \in \Phi} N_G(U_{\alpha}), U_{\pm} := \langle U_{\alpha} \mid \alpha \in \Phi_{\pm} \rangle$:
\begin{enumerate}[label=(RGD\arabic*), leftmargin=*] \setcounter{enumi}{-1}
	\item For each $\alpha \in \Phi$, we have $U_{\alpha} \neq \{1\}$.
	
	\item For each prenilpotent pair $\{ \alpha, \beta \} \subseteq \Phi$, the commutator group $[U_{\alpha}, U_{\beta}]$ is contained in the group $U_{(\alpha, \beta)} := \langle U_{\gamma} \mid \gamma \in (\alpha, \beta) \rangle$.
	
	\item For every $s\in S$ and each $u\in U_{\alpha_s} \backslash \{1\}$, there exists $u', u'' \in U_{-\alpha_s}$ such that the product $m(u) := u' u u''$ conjugates $U_{\beta}$ onto $U_{s\beta}$ for each $\beta \in \Phi$.
	
	\item For each $s\in S$, the group $U_{-\alpha_s}$ is not contained in $U_+$.
	
	\item $G = H \langle U_{\alpha} \mid \alpha \in \Phi \rangle$.
\end{enumerate}

It is well-known that any RGD-system $\mathcal{D}$ acts on a twin building, which is denoted by $\Delta(\mathcal{D})$ (cf. \cite[Section $8.9$]{AB08}). This twin building is a so-called \textit{Moufang twin building} (cf. \cite[Section $8.3$]{AB08}). There is a distinguished pair of opposite chambers in $\Delta(\mathcal{D})$, which we will denote by $(c_+, c_-)$.

\begin{lemma}\label{AB08Cor8.32}
	For $\epsilon \in \{+, -\}$ the group $U_{\epsilon}$ acts simply transitively on the set of chambers opposite $c_{\epsilon}$.
\end{lemma}
\begin{proof}
	This is \cite[Corollary $8.32$]{AB08}.
\end{proof}

We say that an RGD-system $\mathcal{D} = \left( G, \left( U_{\alpha} \right)_{\alpha \in \Phi}\right)$ is \textit{wall-connected}, if the following condition is satisfied
\begin{equation}
\forall (\epsilon, s)\in \{+, -\} \times S: U_{\epsilon} = \langle U_{\beta} \mid \beta \in \Phi_{\epsilon}, o(r_{\beta}s) < \infty \rangle \tag{$\mathrm{wc}$}
\end{equation}

For the notion of \textit{twin roots} we refer to \cite[Section $5.8.5$]{AB08}. Let $\alpha$ be a twin root. Then we define the \textit{wall} associated to $\alpha$ by the set of all panels $P$ such that $P$ is stabilized by $r_{\alpha}$.

\begin{lemma}\label{parallelpanelssamewall}
	Let $\epsilon \in \{+,-\}, P \subseteq \Cp, Q \subseteq \Cm$ be two parallel panels and let $s\in S$ be the type of $P$. Then the reflection $s$ stabilizes $P$ and $Q$.
\end{lemma}
\begin{proof}
	By Theorem \ref{parallelPcomppath} there exists a $P$-compatible path $(Q_0, \ldots, Q_k = Q)$. Since $P$ and $Q_0$ are opposite, both panels are stabilized by the reflection $s$. The claim follows by induction and the fact, that opposite panels in a rank $2$ residue are stabilized by the same reflection.
\end{proof}

\subsubsection*{A criterion for wall-connectedness}

\begin{lemma}\label{Wstransitive}
	Let $s\in S, \epsilon \in \{+, -\}$, let $P := \P_s(c_{\epsilon}) \subseteq \C_{\epsilon}$ and let $P_0, \ldots, P_k \subseteq \C_{-\epsilon}$ be panels such that $(P_0, \ldots, P_k)$ is a $P$-compatible path. Then the group $\langle U_{\beta} \mid \beta \in \Phi_{\epsilon}, o(r_{\beta}s) < \infty \rangle$ acts transitively on the set of panels opposite $P_k$ in $R(P_{k-1}, P_k)$.
\end{lemma}
\begin{proof}
	For $s\in S$ we abbreviate $W_s := \langle U_{\beta} \mid \beta \in \Phi_{\epsilon}, o(r_{\beta}s) < \infty \rangle$. Since $P_{k-1}, P_k$ are opposite in $R(P_{k-1}, P_k)$, it suffices to show that for any panel $Q \subseteq R(P_{k-1}, P_k)$ which is opposite to $P_k$ in $R(P_{k-1}, P_k)$ there exists $g\in W_s$ such that $g.Q = P_{k-1}$. Let $Q$ be such a panel. Then there exists $y\in Q$ such that $\proj_{P_k} c_{\epsilon}, y$ are opposite in $R(P_{k-1}, P_k)$. Let $x\in P_{k-1}$ be opposite to $\proj_{P_k} c_{\epsilon}$ in $R(P_{k-1}, P_k)$. We will show that there exists $g\in W_s$ such that $g.y = x$. Let $(c_0 = \proj_{P_k} c_{\epsilon}, \ldots, c_k = x)$ and $(d_0 = \proj_{P_k} c_{\epsilon}, \ldots, d_k = y)$ be minimal galleries and let $i = \max\{ 0 \leq j \leq k \mid \forall 0 \leq k \leq j: c_k = d_k \}$. We will show the hypothesis by induction on $k-i$. If $k-i = 0$ there is nothing to show. Now let $k-i >0$. Let $\beta$ be the twin root such that $c_i \in \beta, c_{i+1} \notin \beta$. Then $c_{\epsilon} \in \beta$. Since the twin building is a Moufang twin building, there exists $g\in U_{\beta}$ such that $g.d_{i+1} = c_{i+1}$ (cf. \cite[Example $8.47$]{AB08}). Since $o(r_{\beta} s) < \infty$ by Lemma \ref{parallelpanelssamewall} we have $g\in W_s$. By induction we obtain $h\in W_s$ such that $hg.y = h.(g.y) = x$. This finishes the claim.
\end{proof}

\begin{theorem}\label{Theorem1}
	Let $\mathcal{D} = \left( G, \left( U_{\alpha} \right)_{\alpha \in \Phi}\right)$ be an RGD-system of type $(W, S)$ and let $(\epsilon, s) \in \{+, -\} \times S$. Then the following are equivalent:
	\begin{enumerate}[label=(\roman*)]
		\item $U_{\epsilon} = \langle U_{\beta} \mid \beta \in \Phi_{\epsilon}, o(r_{\beta}s) < \infty \rangle$;
		\item $\Gamma_s(c_{\epsilon})$ is connected.
	\end{enumerate}
\end{theorem}
\begin{proof}
	Again, we abbreviate $W_s := \langle U_{\beta} \mid \beta \in \Phi_{\epsilon}, o(r_{\beta}s) < \infty \rangle$. At first we assume that $\Gamma_s(c_{\epsilon})$ is connected. Let $x, y \in c_{\epsilon}^{\mathrm{op}}$ such that $\P_s(x), \P_s(y)$ are wall-adjacent of type $(c_{\epsilon}, s)$. It suffices to show that there exists $g\in W_s$ such that $g.x = y$ (the general case follows by induction and Lemma \ref{AB08Cor8.32} applied to $y = h.x$ for $h \in U_{\epsilon}$). Let $(P_0 = \P_s(x), \ldots, P_k = T)$ and $(Q_0 = \P_s(y), \ldots Q_k = T)$ be $\P_s(c_{\epsilon})$-compatible paths of the same length and type. We show the hypothesis via induction on $k$. If $k=0$ we have $\P_s(x) = \P_s(y)$ and therefore $\delta_{-\epsilon}(x, y) \in \langle s \rangle$. Then there exists $g\in U_{\alpha_s}$ ($\alpha_s$ is the twin root containing $c_{\epsilon}$ but not any $s$-adjacent chamber) with $g.x = y$. Now let $k>0$. By Lemma \ref{Wstransitive} there exists $g\in W_s$ such that $g.P_{k-1} = Q_{k-1}$. We obtain the $\P_s(c_{\epsilon})$-compatible paths $(g.P_0, \ldots, g.P_{k-1} = Q_{k-1})$ and $(Q_0, \ldots, Q_{k-1})$. Using induction we obtain $h\in W_s$ such that $hg.x = h.(g.x) = y$.
	
	Now we assume that $U_{\epsilon} = W_s$. Let $\beta \in \Phi_{\epsilon}$ such that $o(r_{\beta}s)<\infty$ and let $g \in U_{\beta}$. Since $o(r_{\beta}s) < \infty$, there exists $1 \leq k\in \NN$ such that $(r_{\beta}s)^k =1$. Then $(r_{\beta}s)^{k-1}r_{\beta} = s$ and hence we have $v^{-1} tv = s$ for some $v\in W$ and $t\in \{r_{\beta}, s\}$ Since $r_{\beta}$ is a reflection, we have $r_{\beta} = w^{-1}uw$ for some $w\in W$ and $u\in S$. In particular, we have $v^{-1}sv = r$ for some $v\in W, r\in S$. Note that $(sv)^{-1}s(sv) = r$. Thus let $v' \in \{ v, sv \}$ be such that $\ell(sv') = \ell(v')+1$. Let $z\in A_{-\epsilon}(c_+, c_-)$ be such that $\delta_{-\epsilon}(c_{-\epsilon}, z) = v'$. Then $\P_s(c_{-\epsilon})$ and $\P_r(z)$ are parallel by Lemma \ref{DMVMLemma14}. Since $\ell(v'r) = \ell(sv') = \ell(v') +1$, we deduce $z = \proj_{\P_r(z)} c_{-\epsilon}$. By Lemma \ref{DMVMLemma19} there exists a compatible path $(P_0 = \P_s(c_{-\epsilon}), \ldots, P_n = \P_r(z))$. This compatible path is $\P_s(c_{\epsilon})$-compatible by Theorem \ref{P-compatiblepath}. Since $\delta(\P_s(c_{-\epsilon}), \P_r(z)) = \delta(\P_s(g.c_{-\epsilon}), \P_r(z))$ we have also a $\P_s(c_{\epsilon})$-compatible path $(Q_0 = \P_s(g.x), \ldots, Q_n = \P_r(z))$ by Lemma \ref{DMVM11Lemma26} of the same length and type. Hence $\P_s(c_{-\epsilon})$ and $\P_s(g.c_{-\epsilon})$ are wall-adjacent of type $(c_{\epsilon}, s)$. Using induction the claim follows from Lemma \ref{AB08Cor8.32}.
\end{proof}

\begin{corollary}\label{Corollary: D wc iff Delta(D) wc}
	Let $\mathcal{D}$ be an RGD-system of type $(W, S)$. Then the following are equivalent:
	\begin{enumerate}[label=(\roman*)]
		\item $\mathcal{D}$ is wall-connected.
		
		\item $\Delta(\mathcal{D})$ is wall-connected.
	\end{enumerate}
\end{corollary}
\begin{proof}
	Using the fact that $G$ acts transitive on the set of chambers in one half, the claim follows from Lemma \ref{Lemma:Wall-connected survives under isometries} and the previous theorem.
\end{proof}

\subsubsection*{A result about Moufang polygons}

We need a result about Moufang polygons. The proof communicated to us by Richard Weiss relies on the basic theory of Moufang polygons as developed in the first chapters of \cite{TW02}. In this subsection we use the notation of \cite{TW02}.

For $i < k < j, a_i \in U_i, a_j \in U_j$ there exists $a_l \in U_l$, $i < l < j$, such that $[a_i, a_j] = a_{i+1} \cdots a_{j-1}$. We define $[a_i, a_j]_k := a_k$ as well as $[U_i, U_j]_k := \{ [a_i, a_j]_k \mid a_i \in U_i, a_j \in U_j \}$.

\begin{proposition}\label{Wurzelgruppen}
	For each $i+1 \leq k \leq i+n-2$ we have $[U_i, U_{i+n-1}]_k = U_k$.
\end{proposition}
\begin{proof}
	By definition it suffices to show that $U_k \subseteq [U_i, U_{i+n-1}]_k$. We notice that \cite[$(6.4)$]{TW02} is also correct if we shift the indices. We will show the hypothesis by induction on $k-i$. Let $k-i = 1$ and let $a_{i+1} \in U_{i+1}$. By \cite[$(6.1)$]{TW02} we have $U_{i+n-1}^{\mu(a_i)} = U_{i+1}$, where $\mu$ is the mapping defined in \cite[$(6.1)$]{TW02}. Thus there exists $a_{i+n-1} \in U_{i+n-1}$ such that $a_{i+1} = a_{i+n-1}^{\mu(a_i)}$. For each $i < j < i+n-1$ let $b_j \in U_j$ such that $[a_i, a_{i+n-1}^{-1}] = b_{i+1} \cdots b_{i+n-2}$. By \cite[$(6.4)(i)$]{TW02} we have $b_{i+1} = a_{i+n-1}^{\mu(a_i)} = a_{i+1}$ and therefore $[a_i, a_{i+n-1}]_{i+1} = a_{i+1}$. Now let $k-i>1$. Using \cite[$(6.4)(iii)$]{TW02} we obtain $[U_i, U_{i+n-1}]_k = [U_{i+1}, U_{i+n}]_k$ for each $i+2 \leq k \leq i+n-2$. Using induction the claim follows.
\end{proof}

\begin{corollary}
	Let $i+1 \leq k \leq i+n-2$. Then $U_k \leq \langle U_i, U_{i+1}, \ldots, U_{k-1}, U_{k+1}, \ldots, U_{i+n-1} \rangle$.
\end{corollary}
\begin{proof}
	This is a direct consequence of the previous proposition.
\end{proof}

\subsection*{Affine twin buildings of rank $3$}

\begin{proposition}\label{affinerank3}
	Let $\mathcal{D} = \left( G, (U_{\alpha})_{\alpha \in \Phi} \right)$ be an RGD-system of irreducible affine type and of rank $3$. Then $\mathcal{D}$ is wall-connected.
\end{proposition}
\begin{proof}
	We argue in the geometric realization of the Coxeter complex associated with $(W,S)$ in the euclidean plane (as visualized in \cite[Figures $2.1-2.3$]{We09}. Thus we think of the Coxeter complex $\Sigma$ as tessellation of the euclidean plane by chambers (i.e. the triangles). The walls of $\Sigma$ correspond to the lines and each wall determines two halfplanes and these correspond to  the roots of $\Sigma$. We choose a chamber $c$ and identify the set of fundamental reflections with the reflection of the euclidean plane whose wall is a wall of c. Moreover, the set of positive roots $\Phi_+$ is identified with the set of halfspaces that contain $c$. Let $s \in S$. By definition it suffices to show that $U_{\gamma} \subseteq U'$ for each $\gamma \in \Phi_+$, where $U':= \langle U_{\beta} \mid \beta \in \Phi_+, o(sr_{\beta}) <\infty \rangle$. Let $\gamma$ be a root in $\Phi_+$. If $o(sr_{\gamma}) < \infty$, then $U_{\gamma} \subseteq U'$ by the definition of $U'$. Thus, it remains to consider the case where $o(sr_{\gamma}) = \infty$. We consider a gallery $(c=c_0,...,c_{k-1},c_k)$ in $\Sigma$ such that $r_{\gamma}$ switches $c_{k-1}$ and $c_k$ and such that $k$ is minimal for this property. As $o(sr_{\gamma}) = \infty$, we have $k \geq 2$ and therefore a unique rank $2$ residue $R$ containing the chambers $c_{k-2},c_{k-1}$ and $c_k$. We put $d := \proj_R c$ and remark that the wall of $\gamma$ is not a wall of $d$ by the minimality of $k$. In particular, the gonality $m$ of the residue $R$ is at least $3$. The residue $R$ corresponds to a vertex $v$ in the geometric realization of $\Sigma$ and we let $\Phi_+^v$ denote the set of all positive roots having $v$ on their boundary. Let $\alpha$ and $\beta$ be the two roots in $\Phi_+^v$ such that $\{d\} = \alpha \cap \beta \cap R$. Then $\alpha \neq \gamma \neq \beta$ and we have a natural numbering $(\alpha = \alpha_1, \alpha_2, \ldots, \alpha_m = \beta)$ of $\Phi_+^v$ and $1 < \ell < m$ such that $\gamma = \alpha_{\ell}$. Furthermore, we have $o(sr_{\alpha_i}) < \infty$ for all $1 \leq i \leq m$ with $i \neq \ell$. Therefore we have $U_{\alpha_i} \subseteq U'$ for all $1 \leq i \leq m$ with $i \neq \ell$ by the previous case. Thus, it follows from the previous corollary that $U_{\gamma} \subseteq U'$. To show that $U_- = \langle U_{\beta} \mid \beta \in \Phi_-, o(sr_{\beta}) < \infty \rangle$ follows in a similar fashion.
\end{proof}

\begin{lemma}\label{Lemma: Isomorphic foundations implies wc iff wc}
	Let $\Delta = (\Delta_+, \Delta_-, \delta_*), \Delta' = (\Delta_+', \Delta_-', \delta_*')$ be two thick, $2$-spherical twin buildings of rank $\geq 3$, let $c\in \C_+, c' \in \C_+'$ and let $\phi: E_2(c) \to E_2(c')$ be an isometry. Then $\Delta$ is wall-connected if and only if $\Delta'$ is wall-connected.
\end{lemma}
\begin{proof}
	By \cite[Proposition $7.1.6$]{WeDiss21} there exist chambers $d \in c^{\mathrm{op}}$ and $d' \in c'^{\mathrm{op}}$ such that the mapping $d \to d'$ extends the isometry $\phi$ to an isometry $\psi: E_2(c) \cup \{d\} \to E_2(c') \cup \{d'\}$. If $\Delta$ is wall-connected, then this isometry extends to an isometry of the whole twin buildings by Corollary \ref{Cor: Main result}. Now the claim follows from Lemma \ref{Lemma:Wall-connected survives under isometries}. If $\Delta'$ is wall-connected, then the isometry $\psi^{-1}$ extends to an isometry of the whole twin buildings. Again, Lemma \ref{Lemma:Wall-connected survives under isometries} implies that $\Delta$ is wall-connected, too.
\end{proof}

\begin{convention}
	We label the diagrams $\tilde{C}_2$ and $\tilde{G}_2$ in a linear order by $1, 2, 3$ such that $o(s_1 s_2) = 3$ in the case of $\tilde{G}_2$.
\end{convention}

\begin{lemma}\label{Lemma: RtoR extends to foundation}
	Let $\Delta, \Delta'$ be two twin buildings of the same type $\tilde{C}_2$ (resp. $\tilde{G}_2$). Suppose that the $\{s_1, s_2\}$-residues of $\Delta$ and $\Delta'$ are isomorphic to the building associated with $C_2(2)$ (resp. $A_2(2)$ or $A_2(3)$). Let $c\in \Delta, c' \in \Delta'$ be chambers and let $R$ and $R'$ denote the $\{s_2, s_3\}$-residues containing $c$ and $c'$ respectively. Then each isometry from $R$ to $R'$ extends to an isometry from $E_2(c)$ to $E_2(c')$.
\end{lemma}
\begin{proof}
	We shall need the following elementary observation: 
	
	\noindent \textbf{Observarion:} Let $\Gamma$ be one of the buildings associated with $A_2(2), A_2(3)$ or $C_2(2)$ and let $P$ be a panel of $\Gamma$. Then the stabilizer of $P$ in the full automorphism group of $\Gamma$ induces all permutations on the set of chambers in $P$.
	
	The isometry $\phi: R\to R'$ induces an isometry $\P_{s_2}(c) \to \P_{s_2}(c')$. By the observation there exists an isometry $\psi: R_{\{s_1, s_2\}}(c) \to R_{\{s_1, s_2\}}(c')$ as both residues are isomorphic to the building associated to one of the groups $A_2(2), A_2(3), C_2(2)$. The claim follows.
\end{proof}

\begin{lemma}\label{Lemma: C2tilde wall-connected}
	Let $\Delta$ be a twin building of type $\tilde{C}_2$ such that the $\{ s_1, s_2 \}$-residues are isomorphic to the buildings associated with $C_2(2)$. Then $\Delta$ is wall-connected.
\end{lemma}
\begin{proof}
	The $\{ s_2, s_3 \}$-residues are Moufang quadrangles by \cite[$(8.3)$ Theorem $4$]{Ro00} and since the $s_2$-panels have to contain precisely $3$ chambers, the $\{s_2, s_3\}$-residues are all isomorphic to $C_2(2)$ or they are all isomorphic to the unique Moufang quadrangle of order $(2, 4)$. Let $c$ be a chamber of $\Delta$. By (the proof of) \cite[Proposition $4$]{MvM98} and Lemma \ref{Lemma: RtoR extends to foundation}, there exists in both cases an RGD-system $\mathcal{D}$ of type $\tilde{C}_2$, a chamber $c'$ of $\Delta(\mathcal{D})$ and an isometry $\varphi: E_2(c) \to E_2(c')$. Now the claim follows from Proposition \ref{affinerank3}, Corollary \ref{Corollary: D wc iff Delta(D) wc} and Lemma \ref{Lemma: Isomorphic foundations implies wc iff wc}.
\end{proof}

\begin{lemma}\label{Lemma: G2tilde wall-connected}
	Let $\Delta$ be a twin building of type $\tilde{G}_2$ such that the $\{ s_2, s_3 \}$-residues are isomorphic to the building associated with $G_2(2)$ or $G_2(3)$. Then $\Delta$ is wall-connected.
\end{lemma}
\begin{proof}
	The $\{ s_1, s_2 \}$-residues are Moufang planes by \cite[$(8.3)$ Theorem $4$]{Ro00} and since the panel contain precisely $3$ (resp. $4$) chambers, the $\{ s_1, s_2 \}$-residues are all isomorphic to the building associated with $A_2(2)$ (resp. $A_2(3)$). Let $c$ be a chamber in $\Delta$. By (the proof of) \cite[Proposition $4$]{MvM98} and Lemma \ref{Lemma: RtoR extends to foundation} there exists an RGD-system $\mathcal{D}$ of type $\tilde{G}_2$, a chamber $c'$ in $\Delta(\mathcal{D})$ and an isometry $\varphi: E_2(c) \to E_2(c')$. Now the claim follows from Proposition \ref{affinerank3}, Corollary \ref{Corollary: D wc iff Delta(D) wc} and Lemma \ref{Lemma: Isomorphic foundations implies wc iff wc}.
\end{proof}

\begin{theorem}
	Let $\Delta$ be a thick irreducible twin building of affine type $(W, S)$ and of rank $3$. Then $\Delta$ is wall-connected.
\end{theorem}
\begin{proof}
	If there is no rank $2$ residue of $\Delta$ which is isomorphic to $C_2(2), G_2(2)$ or $G_2(3)$, then $\Delta$ satisfies Condition $\co$ by \cite[Section $1$]{MR95} and is therefore wall-connected by Corollary \ref{coimplieswc}. If there is a residue isomorphic to $C_2(2)$, then $\Delta$ is wall-connected by \cite[Corollary $5.157$]{AB08} and Lemma \ref{Lemma: C2tilde wall-connected} and if there is a residue isomorphic to $G_2(2)$ or $G_2(3)$, then $\Delta$ is wall-connected by \cite[Corollary $5.157$]{AB08} and Lemma \ref{Lemma: G2tilde wall-connected}.
\end{proof}

\bibliography{references}
\bibliographystyle{abbrv}

\end{document}